\documentclass{article}

\usepackage[english]{babel}
\usepackage{amsthm}
\usepackage[T1]{fontenc}
\usepackage[utf8]{inputenc}
\usepackage[english]{babel}
\usepackage{graphicx, color}
\usepackage{amssymb}
\usepackage{amsmath, xfrac}
\usepackage{latexsym}
\usepackage{caption, subcaption}
\usepackage{lineno}
\usepackage{comment}
\usepackage{mathtools}  
\usepackage{enumitem} 
\usepackage{hyperref}
\usepackage{cite}

\newcounter{contSect} \numberwithin{contSect}{section}
 \numberwithin{contSub}{subsection}

\newtheorem{theorem}[contSect]{Theorem}
\newtheorem{corollary}[contSect]{Corollary}
\newtheorem{lemma}[contSect]{Lemma}

\newtheorem{conjecture}[contSect]{Conjecture}
\newtheorem{claim}[contSect]{Claim}

\newtheorem{observation}[contSect]{Observation}
\newtheorem{problem}[contSect]{Problem}

\usepackage[letterpaper,top=2cm,bottom=2cm,left=3cm,right=3cm,marginparwidth=1.75cm]{geometry}

\title{An algorithm for estimating the crossing number of dense graphs, and continuous analogs of the crossing and rectilinear crossing numbers}
\author{Oriol Solé-Pi\footnote{Department of Mathematics, Massachusetts Institute of Technology. \textit{Email}: oriolsp@mit.edu}}
\date{}

\begin{document}
\maketitle

\begin{abstract}
We present a deterministic $n^{2+o(1)}$-time algorithm that approximates the crossing number of any graph $G$ of order $n$ up to an additive error of $o(n^4)$. We also provide a randomized polynomial-time algorithm that constructs a drawing of $G$ with $\text{cr}(G)+o(n^4)$ crossings. These results yield a $1+o(1)$ approximation algorithm for the crossing number of dense graphs. Our work complements a paper of Fox, Pach and Súk~\cite{approximatingrect}, who obtained similar results for the rectilinear crossing number.

The results in~\cite{approximatingrect} and in this paper imply that the (normalized) crossing and rectilinear crossing numbers are estimable parameters. Motivated by this, we introduce two graphon parameters, the \textit{crossing density} and the \textit{rectilinear crossing density}, and we prove that, in a precise sense, these are the correct continuous analogs of the crossing and rectilinear crossing numbers of graphs.\end{abstract}  

\section{Introduction}\label{sec:intro}

We work with finite, simple and undirected graphs.

Let $G=(V,E)$ be a graph. A \textit{drawing} of $G$ is a representation in which the vertices are mapped to distinct points on the plane and the edges are represented by simple continuous curves connecting their respective endpoints. We further assume that no edge goes through a vertex other than its endpoints, no two edges are tangent at any point, and no three of them have an interior point in common. A \textit{crossing} is a common interior point of two edges in a drawing. The \textit{crossing number} of $G$, denoted by $\mathop{\mathrm{cr}}(G)$, is the minimum number of crossing points between edges when the minimum is taken over all drawings of $G$. Note that any drawing of $G$ with $\mathop{\mathrm{cr}}(G)$ crossings has the additional property that no two edges cross more than once and no two adjacent edges cross. A \textit{straight-line drawing} of $G$ is a drawing such that each edge is represented by a segment joining the corresponding endpoints. The \textit{rectilinear crossing number} of $G$, $\overline{\mathop{\mathrm{cr}}}(G)$, is the least number of crossings amongst all straight-line drawings of $G$. Clearly, $\mathop{\mathrm{cr}}(G)\leq\overline{\mathop{\mathrm{cr}}}(G)$, and it is known that there are graphs for which the inequality is strict (a rather surprising example was obtained by Bienstock and Dean~\cite{cr4largecr}, who constructed graphs with crossing number $4$ but arbitrarily large rectilinear crossing number).

The crossing number and the rectilinear crossing number have been studied extensively, and we refer the reader to the surveys of Schaefer~\cite{schaefersurvey} and Pach and Tóth~\cite{thirteenproblems} for a review of the existing literature and several interesting questions. One of the central open problems in the area is the determination of the asymptotic behaviour of $\overline{\mathop{\mathrm{cr}}}(K_n)$; while it is well known that the limit \[\lim_{n\rightarrow\infty}\frac{\overline{\mathop{\mathrm{cr}}}(K_n)}{\binom{n}{4}}\] exists (cf.~\cite{sylvester}), finding it has proven to be very challenging. Currently, the best known bounds place this quantity between $0.379972$ and $0.380473$; these are due to Ábrego et al.~\cite{recKnlower} and Fabila-Monroy and López~\cite{recKnupper}, respectively. For the crossing number, it is conjectured that \[\mathop{\mathrm{cr}}(K_n)=\frac{1}{4}\lfloor\frac{n}{2}\rfloor\lfloor\frac{n-1}{2}\rfloor\lfloor\frac{n-2}{2}\rfloor\lfloor\frac{n-3}{2}\rfloor,\] and drawings with this number of crossings have been known for several years (c.f. Moon~\cite{mooncrKn}, Guy~\cite{guycrKn}), but a proof that this is optimal has remained elusive. Still, the limit \[\lim_{n\rightarrow\infty}\frac{\mathop{\mathrm{cr}}(K_n)}{\binom{n}{4}}\] exists and, as evidenced by the aforementioned constructions, is bounded from above by $3/8$. An interesting consequence of this is that this limit differs from the one for rectilinear crossing numbers. The asymptotic behaviours of $\mathop{\mathrm{cr}}(K_{m,n})$ and $\overline{\mathop{\mathrm{cr}}}(K_{m,n})$ are not completely understood either (see~\cite{schaefersurvey} and the references therein). The rectilinear crossing number of $K_n$ is closely related to $k$-sets and $k$-edges\footnote{Given a finite set $P$ of points on the plane, a $k$\textit{-set} is a $k$-element subset $S\subset P$ for which there exists a half-plane $H$ with $H\cap P=S$, and a $k$\textit{-edge} consists of a pair of points $p,q\in P$ such that one of the closed half-planes determined by the line trough $p$ and $q$ contains exactly $k$ points of $P$.}~\cite{k-setslovasz,k-setsbalogh}, as well as to Sylvester's four point problem\footnote{Sylvester's four point problem asks for the probability that four points chosen at random from a region $R$ of the plane lie in convex position. When $R$ is the whole plane, there are a few natural probability distributions from which one can choose the points; this led to some different answers being published. See~\cite{historySylvester} for an overview of the history of the problem.}~\cite{sylvester}. We shall come back to the later of these connections at the end of Section~\ref{sec:graphons}.

Moving on to the computational aspects of the problem, computing the crossing number is known to be NP-complete~\cite{crNP}, while determining the rectilinear crossing number is complete for the existential theory of reals~\cite{cromplexityrectilinear} (and hence NP-hard). In fact, there is some $c>0$ such that approximating the crossing number of $G$ up to a factor of $1+c$ is NP-hard, even for cubic graphs~\cite{cabello2013hardness}. However, for any fixed $k$, there is a linear time algorithm that decides whether $\mathop{\mathrm{cr}}(G)\leq k$~\cite{kawarabayashi2007computing} (in particular, the crossing number is fixed parameter tractable). A considerable amount of work has been put into developing approximation algorithms for both $\mathop{\mathrm{cr}}(G)$ and $\overline{\mathop{\mathrm{cr}}}(G)$. A graph drawing technique of Bhatt and Leighton~\cite{bhattleighton}, in conjunction with the results of Leighton and Rao~\cite{multicommodity} on balanced cuts, can be combined to find, in polynomial time, a straight-line drawing of any bounded degree $n$-vertex graph $G$ with no more than $O(\log^4 n(n+\mathop{\mathrm{cr}}(G))$ crossings. This was later improved to $O(\log^3 n(n+\mathop{\mathrm{cr}}(G))$ by Even et al.~\cite{VLSI}, and then to $O(\log^2 n(n+\mathop{\mathrm{cr}}(G))$ as a result of the improved approximation algorithm for optimal balanced cuts by Arora et al.~\cite{BalancedCut}. It wasn't until several years later that Chuzhoy~\cite{n^9/10approx}, using the edge planarization method\footnote{A subset of the edges of a graph is called \textit{planarizing} if deleting each of its elements results in a planar graph. The method can be summarized as follows: First, one tries to find a small planarizing set of edges and computes a planar drawing of the graph that we get after deleting those edges. Then, the edges of the planarizing set are carefully added to the drawing one by one until we obtain a drawing of the original graph.} from~\cite{planarizationapprox}, found a polynomial-time $O(n^{9/10})$-approximation algorithm for $\mathop{\mathrm{cr}}(G)$ for bounded degree graphs. Building further on the edge planarization method, Kawarabayashi and Sidiropoulos~\cite{kawarabayashi1,kawarabayashi2} improved the approximation ratio to $O(n^{1/2})$, and then Mahabadi and Tan~\cite{1/2-epsilon} found a randomized $O(n^{1/2-\delta})$-approximation algorithm, where $\delta>0$ is a constant.

The celebrated crossing lemma, discovered simultaneously by Ajtai at al.~\cite{crossinglemma} and Leighton~\cite{crossinglemma2}, tells us that $\mathop{\mathrm{cr}}(G)\geq\frac{|E|^3}{64|V|^2}$, so long as $|E|\geq 4|V|$. An immediate consequence of this result is that if $G$ is dense (that is, it has $\Omega(|V|^2)$ edges) then both $\mathop{\mathrm{cr}}(G)$ and $\overline{\mathop{\mathrm{cr}}}(G)$ are $\Omega(|V|^4)$. Fox, Pach and Suk~\cite{approximatingrect} presented an algorithm that constructs a straight-line drawing of $G$ with $\overline{\mathop{\mathrm{cr}}}(G)+o(|V|^4)$ crossings. More precisely, they showed the following. 

\begin{theorem}\label{teo:FoxPachSuk}
    There is a deterministic $n^{2+o(1)}$-time algorithm that computes a straight-line drawing of any given $n$-vertex graph $G$ with no more than \[\overline{\mathop{\mathrm{cr}}}(G)+O(n^4/(\log\log n)^\delta)\] crossings, where $\delta$ is an absolute and positive constant.
\end{theorem}

Note that if $G$ is dense then the number of crossings in the drawing provided by this algorithm is $(1+o(1))\overline{\mathop{\mathrm{cr}}}(G)$. We obtain a similar result for the crossing number.

\begin{theorem}\label{teo:main}
    There exists a deterministic $n^{2+o(1)}$-time algorithm that for any given  $n$-vertex graph $G$ approximates $\mathop{\mathrm{cr}}(G)$ up to an additive error of $O(n^4/(\log\log n)^{\delta'})$. Furthermore, there is a randomized polynomial-time algorithm that, with probability $1-o(1)$, computes a drawing of $G$ with \[\mathop{\mathrm{cr}}(G)+O(n^4/(\log\log n)^{\delta'})\] crossings. Here, $\delta'$ denotes an absolute positive constant.
\end{theorem}

At a high level, the approximation part of the algorithm follows the strategy devised by Fox, Pach and Suk for the proof of Theorem~\ref{teo:FoxPachSuk}. The main novel ingredient that is required for the proof of correctness of this algorithm is Theorem~\ref{teo:closegraphs}, which provides a bound on the difference between the crossing numbers of two graphs of the same order in terms of their distance in the cut norm.

A consequence of the results in~\cite{approximatingrect} and Theorem~\ref{teo:closegraphs} is that the (normalized) crossing and rectilinear crossing numbers are estimable/testable parameters, in the sense of~\cite{borgslimits}. Motivated by this, we define two graphon parameters, which we call the \textit{crossing density} and the \textit{rectilinear crossing density}, by means of continuous analogs of the notions of drawings and straight-line drawings of graphs. We show that both of these parameters are continuous with respect to the cut norm and that, in a precise sense, they behave as the limits of the crossing number and the rectilinear crossing number of graphs. This discussion is directly tied to some of the problems we mentioned earlier about the asymptotic behaviors of crossing numbers. We hope that our results might prove useful in the study of crossing and rectilinear crossing numbers of dense graphs.

\subsection*{Outline of the paper} The basic definitions regarding graphons, cut distance and estimable parameters, as well as some other necessary preliminaries, are included in Section~\ref{sec:prelim}. In Section~\ref{sec:estimatingcr}, we prove Theorem~\ref{teo:main}; most of this section is in fact devoted to the proof of Theorem~\ref{teo:closegraphs}, which was mentioned above. The crossing and rectilinear crossing densities of graphons are defined and studied in Section~\ref{sec:graphons}. Lastly, we discuss some unanswered questions in Section~\ref{sec:final}.

\section{Preliminaries}\label{sec:prelim}

By an \textit{edge-weighted graph}, we mean a graph $G=(V,E)$ where each edge $(u,v)\in E$ has a weight $w_G(u,v)\in[0,1]$ assigned to it. We write $w_G(u,v)=0$ whenever $(u,v)\notin E$. We shall work mostly with edge-weighted graph, and we often refer to them simply as graphs.

\subsection{Crossing numbers of edge-weighted graphs}\label{sec:weightedcrossingnumber}

Next, we extend the definitions of the crossing and the rectilinear crossing numbers to edge-weighted graphs. For an edge-weighted graph $G(V,E)$ and a drawing $\mathcal{D}$ of $G$, let $C(\mathcal{D})$ denote the multi-set of all pairs of edges that cross each other in the drawing, with the proper multiplicity (i.e., if two edges cross each other at $k$ points then this pair appears $k$ times in $C(\mathcal{D})$). Now, let \[\mathop{\mathrm{cr}}(G,\mathcal{D})=\sum_{(e_1,e_2)\in C(\mathcal{D})}w_G(e_1)w_G(e_2)\] and define $\mathop{\mathrm{cr}}(G)$ as the least value of $\mathop{\mathrm{cr}}(G,\mathcal{D})$ over all drawings of $G$. Similarly, $\overline{\mathop{\mathrm{cr}}}(G)$ is the minimum of $\mathop{\mathrm{cr}}(G,\mathcal{D})$ where $\mathcal{D}$ ranges over all straight-line drawings of $G$. We say that a drawing $\mathcal{D}$ of $G$ \textit{attains} $\mathop{\mathrm{cr}}(G)$ if $\mathop{\mathrm{cr}}(G,\mathcal{D})=\mathop{\mathrm{cr}}(G)$, and that it \textit{attains} $\overline{\mathop{\mathrm{cr}}}(G)$ if it is a straight-line drawing and $\mathop{\mathrm{cr}}(G,\mathcal{D})=\overline{\mathop{\mathrm{cr}}}(G)$. A drawing such that no two edges cross more than once and no two adjacent edges cross will be called \textit{simple}. Although it is not as self-evident as in the unweighted case, one can readily show that there must exist a simple drawing which attains $\mathop{\mathrm{cr}}(G)$. Note that, for unweighted graphs, we can assign a weight of $1$ to every edge in order to recover the definitions of crossing number and rectilinear crossing number provided in the introduction.

We will need the following simple result, which fulfills the same role as Lemma 2 in~\cite{approximatingrect}. It seems very likely that a result of this kind has explicitly appeared somewhere else already, but we have been unable to find it.

\begin{theorem}\label{teo:exactcr}
    Let $G$ be an edge-weighted graph on $n$ vertices such that the weight of each edge can be represented using no more than $B$ bits. Then we can find a drawing of $G$ that attains $\mathop{\mathrm{cr}}(G)$ in $2^{O(n^4\log n)}+2^{O(n^2)}B^2$-time.
\end{theorem}

\begin{proof}
    Given a set $C\subset\binom{E}{2}$, we can determine whether there is a simple drawing $\mathcal{D}$ of $G$ with $C(\mathcal{D})=C$ in the following way: For any given edge $e\in E$, there are less than $n^2!$ possible orders in which the points where $e$ crosses the other edges may appear along the curve representing $e$. For a fixed order of the crossings along every edge, we can test whether it arises from an actual drawing of $G$ by placing a dummy vertex at each crossing and then using a linear time planarity testing algorithm (see~\cite{patrignani2013planarity}). Thus, we can test whether $C$ actually comes from a drawing in $n^2!^{O(n^2)}=2^{O(n^4\log n)}$-time.
  
    If we are given $C(\mathcal{D})$, then we can compute $\mathop{\mathrm{cr}}(G,\mathcal{D})$ in $O(n^2B^2)$-time. Note also that the number of possible $C$'s we must check is $2^{|E|}\leq 2^{\binom{n}{2}}$. Hence, we can find a simple drawing $\mathcal{D}$ of $G$ that attains $\mathop{\mathrm{cr}}(G)$ in $2^{O(n^4\log n)}+O(2^{\binom{n}{2}}n^2B^2)$-time.
\end{proof}

In order to study graphons in section~\ref{sec:graphons}, it will be convenient to have a normalized version of the crossing and rectilinear crossing numbers. With this in mind, for an edge-weighted graph $G$ on $n$ vertices we define its \textit{crossing density} as $\mathop{\mathrm{cd}}(G)=\mathop{\mathrm{cr}}(G)/n^4$ and its \textit{rectilinear crossing density} as $\overline{\mathop{\mathrm{cd}}}(G)=\overline{\mathop{\mathrm{cr}}}(G)/n^4$.

\subsection{Cut distance and graphons}\label{sec:cutdistance}

We expect the reader to be somewhat familiar with the theory of graphons and convergent sequences of dense graphs, but include the required fundamental definitions for the sake of completeness. We refer the reader to the book by Lovász~\cite{lovaszlimits}, which we follow rather closely during for the remainder of the section, for an in-depth treatment of the subject.

Let $G=(V,E)$ be an edge-weighted graph. For any two subsets $S,T\subset V$, let $E_G(S,T)$ denote the set of edges with one endpoint in $S$ and the other one in $T$, and let $e_G(S,T)$ be the total weight of the elements of $E_G(S,T)$, where edges with both endpoints in $S\cap T$ are counted twice. 

Given a positive integer $m$, the \textit{m blow-up} $G[m]$ of $G$ is the edge-weighted graph obtained by replacing each vertex $v$ of $G$ by an independent set $U_v$ with $m$ elements, and then setting the weight of every edge between $U_u$ and $U_v$ (with $u\neq v$) to be $w_G(u,v)$. For any partition $\mathcal{P}=\{V_1,V_2,\dots,V_n\}$ of $V$, let $G_\mathcal{P}$ denote the edge-weighted graph with vertex set $V$ such that \[w_G(u,v)=\frac{e_G(V_i,V_j)}{|V_i||V_j|}\] whenever $u\in V_i$ and $v\in V_j$. Also, let $G/\mathcal{P}$ denote the graph with vertex set $\{1,2,\dots,n\}$ and edge weights \[w_{G/\mathcal{P}}(i,j)=\frac{e_G(V_i,V_j)}{|V_i||V_j|}.\]

The \textit{labeled cut distance} between two edge-weighted graphs $G_1$ and $G_2$ on the same finite vertex set $V$ is defined as \[d_\square(G_1,G_2)=\max_{S,T\subset V}\frac{|e_{G_1}(S,T)-e_{G_2}(S,T)|}{|V|^2}.\](Note: The reader who is only interested in the proof of Theorem~\ref{teo:main} can now skip to subsection~\ref{sec:Frieze-Kannan}.)

If $G_1$ and $G_2$ are are defined on possibly different $n$-element vertex sets, then we write \[\widehat{\delta}_\square(G_1,G_2)=\min_{G_1',G_2'}d_\square(G_1',G_2'),\] where $G_1'$ and $G_2'$ range over all graphs with vertex set $\{1,2,\dots,n\}$ which are isomorphic to $G_1$ and $G_2$, respectively. We are now ready to define the cut distance between two arbitrary graphs. Given two edge-weighted graphs $G_1$ and $G_2$ on $m$ and $n$ vertices, respectively, the \textit{cut distance} between them is given by \[\delta_\square(G_1,G_2)=\lim_{k\rightarrow\infty}\widehat{\delta}_\square(G_1[kn],G_2[km]),\] which can be shown to be well defined. The distance function $\delta_\square$ is a pseudometric on the set of edge-weighted graphs (it is symmetric and satisfies the triangle inequality, but $\delta_\square(G_1,G_2)=0$ does not imply that $G_1$ and $G_2$ are isomorphic). We say that a sequence edge-weighted graphs $G_1,G_2,\dots$ is \textit{convergent} if it is Cauchy with respect to $\delta_\square$. 

Throughout this paper, measurability is always considered with respect to the Lebesgue $\sigma$-algebra, although working with Borel measurability would not make a significant difference. We denote the Lebesgue measure by $\lambda$. A \textit{kernel} is a symmetric measurable function $W:[0,1]^2\rightarrow\mathbb{R}$ (by symmetric, we mean that $W(x,y)=W(y,x)$); the space of all kernels is denoted by $\mathcal{W}$. A \textit{graphon} is a kernel whose image is a subset of $[0,1]$, and we denote the space of graphons by $\mathcal{W}_0$. 

The \textit{cut norm} on $\mathcal{W}$ can be written as \[||W||_\square=\sup_{S,T\subseteq[0,1]}\left|\int_{S\times T} W(x,y)\ dx\ dy\right|,\] and the \textit{labeled cut distance} between two kernels $W_1$ and $W_2$ is $d_\square(W_1,W_2)=||W_1-W_2||_\square$. A measurable function $\phi:[0,1]\rightarrow [0,1]$ is \textit{measure preserving} if $\lambda(S)=\lambda(\phi^{-1}(S))$ for every measurable $S\subseteq[0,1]$. For a kernel $W$ and a measure preserving $\phi:[0,1]\rightarrow[0,1]$, let $W^\phi$ denote the kernel with $W^\phi(x,y)=W(\phi(x),\phi(y))$. Now, for any two kernels $W_1$ and $W_2$, the \textit{cut distance} between them is defined as \[\delta_\square(W_1,W_2)=\inf_{\phi}d_\square(W_1,W_2^\phi),\] where $\phi$ ranges over all invertible measure preserving maps from $[0,1]$ to itself. The distance function $\delta_\square$ is a pseudometric on the space $\mathcal{W}$, and two kernels $W_1$ and $W_2$ are said to be \textit{weakly isomorphic} if $\delta_\square(W_1,W_2)=0$. Thus, $\delta_\square$ induces a metric on the quotient space $\widetilde{\mathcal{W}_0}$ that arises from $\mathcal{W}_0$ after identifying all classes of weakly isomorphic graphons. One of the central results in the theory of graphons is that $(\widetilde{\mathcal{W}_0},\delta_\square)$ is compact (cf.~\cite{compact}).

For every graph $G$ with vertices $v_1,v_2,\dots,v_n$, the graphon $W_G$ is constructed by splitting $[0,1]$ into $n$ intervals $I_1,I_2,\dots,I_n$ of measure $1/n$ and setting $W_G(x,y)=w_G(v_i,v_j)$ for all $x\in I_i,y\in I_j$. To be precise, $W_G$ depends on both $G$ and an ordering of its vertices, but this will not be an issue, since any two such graphons that arise from the same graph are weakly isomorphic. Given any two edge-weighted graphs $G_1$ and $G_2$, we have that $\delta_\square(G_1,G_2)=\delta_\square(W_{G_1},W_{G_2})$. For every convergent graph sequence $G_1,G_2,\dots$, the graphons $W_{G_1},W_{G_2},\dots$ converge, with respect to $\delta_\square$, to some graphon $W$ (or, rather, to the class of $W$ in $\widetilde{\mathcal{W}_0}$), and we say that the graph sequence converges to $W$. Conversely, for every graphon $W$ there exists a convergent graph sequence $G_1,G_2,\dots$ such that $W_{G_1},W_{G_2},\dots$ converges to $W$ with respect to $\delta_\square$.

If $W$ is a kernel and $\mathcal{P}=\{S_1,S_2,\dots,S_n\}$ is a partition of the unit interval into measurable sets, then $W_\mathcal{P}$ denotes the kernel such that \[W_\mathcal{P}(x,y)=\frac{1}{\lambda(S_i)\lambda(S_j)}\int_{S_i\times S_j}W(x,y)\ dx\ dy\] whenever $x\in S_i$ and $y\in S_j$. Let $\mathcal{P}_1,\mathcal{P}_2,\dots$ be a sequence of partitions of $[0,1]$ into measurable sets such that each pair of points of $[0,1]$ lie in different parts for all but a finite number of partitions of the sequence. Then, as $n$ goes to infinity, $W_{\mathcal{P}_n}$ converges to $W$ almost everywhere for every $W\in\mathcal{W}$.

Given a graph $F=(V,E)$ and a kernel $W$, let \[t(F,W)=\int_{[0,1]^V}\prod_{(u,v)\in E}W(x_u,x_v)\prod_{u\in V}dx_u.\]

It is well known that two kernels $W_1$ and $W_2$ are weakly isomorphic if and only if $t(F,W_1)=t(F,W_2)$ for every simple graph $F$. Moreover, a sequence of kernels $W_1,W_2,\dots$ is convergent with respect to $\delta_\square$ if and only if the limit $\lim_{n\rightarrow\infty}t(F,W_n)$ exists for every graph $F$.

\subsection{Estimable parameters}\label{sec:estimable}

By a \textit{graph parameter}, we mean a function that assigns a real number to each graph (or edge-weighted graph) and is constant on each isomorphism class. A graph parameter $f$ is said to be \textit{estimable} if there is another graph parameter $g$, which we call a \textit{test parameter for $f$}, with the following property: For every $\varepsilon>0$ there exist an integer $k$ such that if $G=(V,E)$ is a graph on at least $k$ vertices and $X$ is a random $k$-element subset of $V$, then \[\mathbb P[|f[G]-g(G[X])|>\varepsilon]\leq\varepsilon,\] where $G[X]$ denotes the subgraph of $G$ that is induced by $X$. It is no hard to see that if $f$ is estimable then can always use $g=f$ (cf.~\cite{goldreichtrevisan,borgslimits}). As shown by Borgs et al.~\cite{borgslimits}, the following properties are equivalent for every graph parameter $f$.

\begin{enumerate}[label=(\alph*)]
    \item $f$ is estimable.
    
    \item For every convergent sequence of graphs $G_1,G_2,\dots$, the sequence $f(G_1),f(G_2),\dots$ converges as well.

    \item There exists a functional $\hat{f}$ on $\mathcal{W}_0$ that is continuous with respect to the cut norm and such that $\hat{f}(W_G)-f(G)\rightarrow 0$ as the number of vertices of $G$ goes to infinity. 
\end{enumerate}

If $f$ is estimable, then the functional $\hat{f}$ mentioned in (c) also satisfies $\lim_{n\rightarrow\infty}f(G_n)=\hat f(W)$ whenever $G_1,G_2,\dots$ converges to $W$ and the number of vertices of $G_n$ goes to infinity with $n$. We often refer to a functional on $\mathcal{W}_0$ as a graphon parameter. It was also proven in~\cite{borgslimits} that the following three properties together are equivalent to $f$ being estimable.

\begin{enumerate}[label=(\roman*)]
    \item For every $\varepsilon>0$, there exists an $\varepsilon'$ such that any two graphs $G_1$ and $G_2$ on the same vertex set with $d_\square(G_1,G_2)<\varepsilon'$ satisfy $|f(G_1)-f(G_2)|<\varepsilon$.
    
    \item For every graph $G$, $f(G[m])$ converges as $m$ goes to infinity.

    \item  $|f(G)-f(G\cup K_1)|\rightarrow0$ as the number of vertices of $G$ goes to infinity. Here, $G\cup K_1$ denotes the graph obtained by adding an isolated node to $G$.
\end{enumerate}

This equivalence will come in handy in Section~\ref{sec:estimatingcr}, since these properties are easier to check than (b) and (c) above.

\subsection{The Frieze-Kannan regularity lemma}\label{sec:Frieze-Kannan}
Given a graph $G=(V,E)$, two sets of vertices $S,T\subset V$ and some $\varepsilon>0$, we say that the pair $(S,T)$ is $\varepsilon$-regular if for any $S'\subset S$, $T'\subset T$ \[\left|\frac{e_G(S',T')}{|S'||T'|}-\frac{e_G(S,T)}{|S||T|}\right|\leq\varepsilon.\] According to Szemerédi's regularity lemma~\cite{Szemeredi}, for every $\varepsilon>0$ there exists an $M(\varepsilon)$ such that for every graph we can find an equitable partition of its vertices into no more than $M(\varepsilon)$ parts with the property that all but at most an $\varepsilon$ fraction of the pairs of parts are $\varepsilon$-regular. Szemerédi's regularity lemma is one of the most powerful tools in the study of dense graphs; unfortunately, it is not very practical for algorithmic purposes, since $M(\varepsilon)$ grows extremely fast as $\varepsilon$ goes to $0$ (see~\cite{gowersregularity} for details on the asymptotic behavior of the optimal value of $M(\varepsilon)$). As was done by the authors of~\cite{approximatingrect}, we circumvent this issue by means of a variant of the regularity lemma developed by Frieze and Kannan~\cite{FriKan}. Given a graph $G=(V,E)$, we say that an equitable partition $\mathcal{P}=\{V_1,V_2,\dots,V_n\}$ of $V$ is a \textit{Frieze-Kannan} $\varepsilon$-\textit{regular} partition if $d_\square(G,G_\mathcal{P})\leq\varepsilon$. The algorithmic version of the Frieze-Kannan regularity lemma stated below is due to Dellamonica et al.~\cite{FriKanAlgorithm}.

\begin{lemma}\label{teo:FriezeKannan}
    There exist a deterministic algorithm and an absolute constant $c$ that, for any $\varepsilon>0$ and any $n$-vertex graph $G$, computes a Frieze-Kannan $\varepsilon$-regular partition of $G$ with no more than $2^{\varepsilon^{-c}}$ classes in $2^{2^{\varepsilon^{-c}}} n^2$-time.
\end{lemma}

\subsection{Cycle separators for planar graphs}

The planar separator theorem from~\cite{ungar,liptontarjan} states that for every $n$-vertex planar graph $G=(V,E)$ there exists a partition $V=A\sqcup B\sqcup C$ such that $|B|=O(\sqrt{n})$, $|A|,|C|\leq 2n/3$, and there are no edges between $A$ and $C$. This result is a cornerstone of the study of planar graphs, and several generalizations and variants of it have been discovered over the years. We will make use of the following version, due to Miller~\cite{cycleseparator}.

\begin{theorem}\label{teo:cycleseparator}
    Let $G$ be an embedded triangulation\footnote{A \textit{triangulation} is a maximal planar graph. By an \textit{embedded} planar graph we simply mean the graph together with a drawing of it where no two edges cross. Every face of an embedded triangulation has precisely three vertices on its boundary. We remark that all ways of embedding a triangulation are, in a precise sense, combinatorially equivalent, so the theorem could have also been stated without ever referring to a particular embedding.} with a non-negative weight assigned to each vertex so that the total sum of the weights is $1$. Then, there exists a simple cycle of length at most $L\sqrt{n}$ in $G$ such that the vertices that lie in its interior have total weight at most $2/3$ and the same is true for those that belong to its exterior. Here, $L$ is an absolute positive constant.
\end{theorem}

\subsection{Triangulating with small degrees}\label{teo:triangulation}

In order to make use of Theorem~\ref{teo:cycleseparator}, we need to be able to transform embedded planar graphs into embedded triangulations by adding some nodes and edges. Moreover, it will be important for our purposes that the arising triangulation does not have too many new vertices, and that degrees of its nodes are not too large. The following result, which essentially appears as Lemma 2.1 in~\cite{pach2005crossing}, allows us to do just that.

\begin{lemma}
    Let $G$ be a connected, embedded, $n$-vertex planar graph. Suppose that every vertex of $G$ has degree at most $d$ for some $d\geq 3$. Then, the embedding can be extended to a triangulation with at most $19n$ vertices and all degrees bounded from above by $3d$.
\end{lemma}

\noindent\textit{Remark.} The proof of this lemma is rather simple. In\cite{pach2005crossing}, the result is stated for two-connected graphs embedded in a genus $g$ surface. An inspection of the proof reveals that the two-connectedness is not required for graphs embedded in the plane. 

\section{Estimating the crossing number}\label{sec:estimatingcr}

\subsection{Subdividing drawings}

We say that a drawing $\mathcal{D}$ of an edge-weighted graph $G$ is \textit{locally optimal} if the value of $\mathop{\mathrm{cr}}(G,\mathcal{D})$ cannot be reduced by erasing a single curve that represents an edge and then redrawing the edge in some other way. If there are no edges of weight $0$, then any locally optimal drawing is necessarily simple. We remark, however, that not every locally optimal drawing attains the crossing number of the graph. This section begins with a result about subdividing simple and locally optimal graph drawings, which is somewhat similar in spirit to the classical cutting lemma for line arrangements\footnote{The cutting lemma is a powerful tool in both discrete and computational geometry. In its two dimensional version, it can be stated as follows: Let $\mathcal{L}$ be a finite family of lines on the plane and let $\varepsilon>0$. Then, it is possible to subdivide the plane into $O(1/\varepsilon^2)$ convex regions, none of which is crossed by more than $\varepsilon|\mathcal{L}|$ lines of $\mathcal{L}$. See~\cite{cuttingsapplications1,cuttingsapplications2} for various proofs of this result and its higher dimensional variants, as well as several of their applications.}, and will play a crucial role in our proof that any two graphs which are close with respect to $d_\square$ have similar crossing numbers (see Theorem~\ref{teo:closegraphs}).

\begin{theorem}\label{teo:regions}
    Let $G$ be an $n$-vertex edge-weighted graph and let $\mathcal{D}$ be a simple and locally optimal drawing of $G$. Then, for any $\varepsilon\in(0,1)$ the plane can be subdivided into $O(1/\varepsilon^2)$ closed, connected and interior disjoint regions, each of which has the following properties: 
    \begin{enumerate}[label=(\Roman*)]
        \item No vertex lies on its boundary.
        \item It contains at most $\lceil\varepsilon^2n\rceil$ vertices.
        \item Any vertex of $G$ and any other point which are both contained in the region can be connected by a simple curve that lies completely within the region and does not go through any vertex or crossing, and whose relative interior has no more than $\varepsilon n^2$ intersection points with the curves that represent the edges of $G$ with no endpoint in that same region.
    \end{enumerate} 
\end{theorem}

\begin{proof}
    If $\varepsilon\leq n^{-1/2}$, the subdivision can be obtained by splitting the plane into $n+1$ regions, $n$ of which are very small and contain precisely one vertex each. From now on, we assume that $n^{-1/2}<\varepsilon<1$.
    
    We begin by subdividing the plane into $O(1/\varepsilon^2)$ closed, connected and interior disjoint regions which contain no more than $\lceil\varepsilon^2 n\rceil$ vertices of $G$ each, and whose boundaries, which do not contain any vertex of $G$, have a total of $O(n^2/\varepsilon)$ intersections with with the edges of $\mathcal D$. Here, by an intersection what we really mean is a connected component of the intersection between an edge and the boundary of one of the regions; an edge and the boundary of a region can induce multiple intersections, but the number of intersections between them is always an integer. This will be achieved through repeated applications of Theorem~\ref{teo:cycleseparator} and Lemma~\ref{teo:triangulation}. More precisely, we will build a sequence $R_0,R_1,\dots,R_t$ of regions as follows: 
    
    Let $R_0$ denote the entire plane. By adding a dummy vertex at each crossing point, $\mathcal{D}$ gives rise to an embedded planar graph. Using Lemma~\ref{teo:triangulation}, this planar graph can then be transformed into an embedded triangulation $T_0$. One can readily check that the original planar graph has at most $n^4/4$ vertices, so $T_0$ has at most $19n^4/4< 5n^4$ vertices, each of degree at most $3n$. Assign a weight of $1/n$ to every vertex of $G$ and a weight of $0$ to every other vertex and apply Theorem~\ref{teo:cycleseparator} to $T_0$, thus obtaining a cycle $C_0$ of length less than $\sqrt{5}Ln^2$ whose interior and exterior each contain at most $2n/3$ vertices of $G$. Every vertex of $G$ that belongs to $C_0$ can be assigned to either the interior or the exterior of the cycle so that the number of vertices in any of these two regions plus the number of vertices that have been assigned to it is still no more than $2n/3$. By modifying the triangulation and the cycle $C_0$ as depicted in Figure~\ref{fig:1}, we obtain a new cycle $C_0'$ with less than $\sqrt{5}Ln^4+n\cdot 3n$ nodes which contains no vertex of $G$. Denote the two closed regions induced by this curve as $R_1$ and $R_2$; these contain at most $2n/3$ vertices of $G$ each.

    For every $i$, we write $n_i$ to denote the number of vertices of $G$ that lie in $R_i$. Now, we iteratively subdivide each $R_i$ that contains $n_i>\lceil\varepsilon^2 n\rceil$ vertices of $G$ and has not yet been subdivided into smaller regions; the rest of this paragraph describes how to carry out this subdivision. First, the boundary of $R_i$ and the portion of $\mathcal D$ that is contained in this region are turned into an embedded planar graph by adding a dummy node at every intersection (either between two edges or between the boundary and an edge). Let $m_i$ denote the order of this graph. A face of this embedded planar graph will be called \textit{null} if its interior is contained in the complement of $R_i$. Observe that there is precisely one null face for each connected component of the complement of $R_i$. Once again, Lemma~\ref{teo:triangulation} allows us to transform this planar graph into an embedded triangulation with no more than $19m_i$ nodes, each of degree at most $3n$. Now, we modify the triangulation by deleting all nodes and edges added in the interior of each null face, and then re-triangulating each such face simply by picking one of the vertices on its boundary and connecting it to all others using curves drawn within the face. This is possible, since among the vertices on the boundary of any null face there is at least one (there are at least two, actually\footnote{In general, given an embedded planar graph this will true of every face whose boundary is a simple curve. This fact might look more familiar to the reader in its equivalent, ``inverted'' form: For any set of pairwise interior disjoint diagonals of a convex polygon, there are at least two nodes of the polygon which are not incident to any of the diagonals.}) which is adjacent to only two other vertices from the said boundary, and any such vertex can be used to draw the edges from. Let $T_i$ be the triangulation that is obtained in this manner. As mentioned above, the process we are about to discuss will ensure that no vertex of $G$ ever lies on the boundary of any of the $R_j$'s (in particular, this implies that no vertex of $G$ will ever lie on the boundary of a null face), so it is still true that the degree of every vertex of $G$ in $T_i$ is upper bounded by $3n$. If we assign weight $1/n_i$ to each of the vertices of $G$ in $R_i$ and weight $0$ to all other vertices of $T_i$, then Theorem~\ref{teo:cycleseparator} yields a cycle $C_i$ of length at most $L\sqrt{19m_i}$, where $m_i$ denotes the order of $T_i$. As we did for $C_0$, $T_i$ and $C_i$ can be modified slightly to obtain a cycle $C_i'$ with no more than $L\sqrt{19m_i}+3n_i$ nodes that does not go through any vertex of $G$ and such that the intersections of $R_i$ with both the interior and the exterior of the cycle each contain at most $2n_i/3$ vertices of $G$. From now on, we shall ignore the edges of $C_i'$ that are drawn on the interior of a null face, as well as all nodes of $C_i'$ which are incident to two ignored edges; note that there are at most two of these edges and one such node for each null face. One could try to use $C_i'$ in order to split $R_i$ into smaller parts, but this may cause the number of regions to blow up in the case that $C_i'$ intersects the boundary of $R_i$ at multiple places. Thus, we first modify $C_i'$ as shown in Figure~\ref{fig:2} for every contiguous sequence of nodes and edges that it shares with the boundary of $R_i$ and lies between two edges of $C_i'$ that are draw inside $R_i$. This results in a new $C_i''$ with at most $4(L\sqrt{19m_i}+3n_i)$ nodes, which is then used to split $R_i$ into two smaller connected regions with at most $2n_i/3$ vertices of $G$ each. 
    
\begin{figure}[!htbp]
\centering
\includegraphics[scale=0.55]{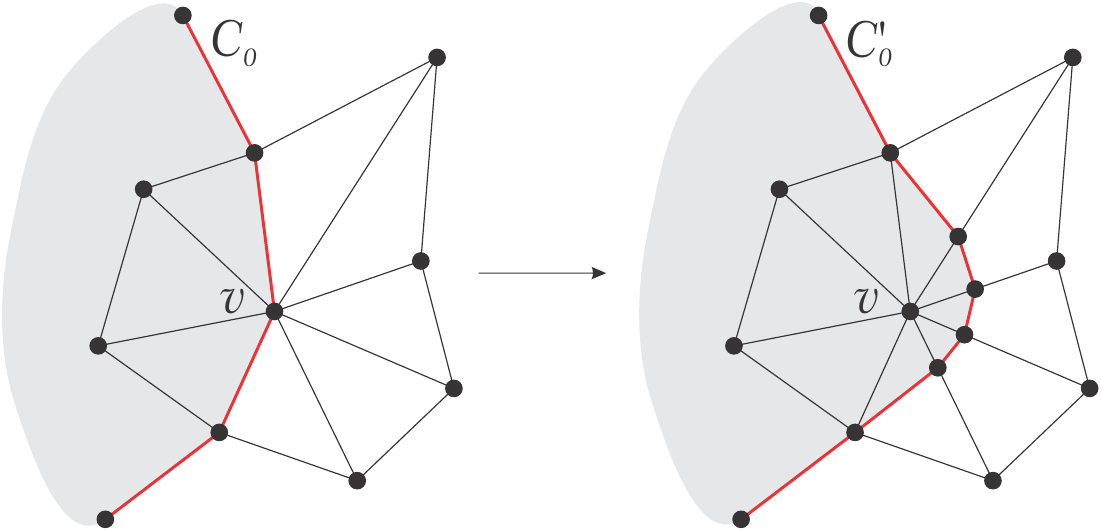}
\caption{A small portion of $C_0$ has been highlighted in red. Assuming that the vertex $v$ of $G$ has been assigned to the shaded region, we alter the triangulation and the cycle around a small neighborhood of $v$ so that this point belongs to the interior of the shaded region defined by $C_0'$. This increases the number of nodes of the cycle (and the whole graph) by less than the degree of $v$ in $T_0$.}
\label{fig:1}
\end{figure}

\begin{figure}[!htbp]
\centering
\includegraphics[scale=0.65]{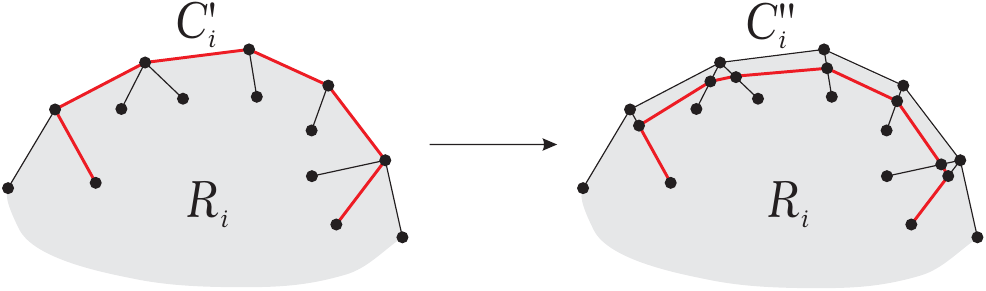}
\caption{The region $R_i$ is shaded. A section of $C_i'$ (red) is contained in the boundary of $R_i$, and the two adjacent arcs of $C_i'$ (also in red) lie inside $R_i$. For every vertex of $T_i'$ in this section, consider the edges of $T_i'$ that are incident to it, lie in $R_i$, and are contained in a curve that represents an edge of $G$ in $\mathcal{D}$. Note that, since no vertex of $G$ belongs to $C_i$, there are at most two edges of $G$ going through any of these vertices. Now, modify $C_i'$ as shown above. This increases the number of nodes in the curve by at most three times the number of nodes that conform the original section of $C_i$.}
\label{fig:2}
\end{figure}
    
    After enough steps, we reach a subdivision of the plane into closed, connected and interior disjoint regions, each of which contains at most $\lceil\varepsilon^2n\rceil$ vertices of $G$. These regions will be referred to as \textit{fundamental regions} from now on. An important feature of above procedure that will prove useful later on is that each connected component of the boundary of a region $R_j$ contains one of the $C_i''$'s in its entirety (possibly, but not necessarily, $C_j''$ itself). The number of fundamental regions is clearly $O(1/\varepsilon^2)$, but we also require an upper bound on the total number of intersections between their boundaries and the edges of $G$ in the drawing. A quick inspection of the subdivision process reveals that this quantity is at most twice the sum of the numbers of nodes in $C_0'$ and in each of the $C_i''$'s. The counting argument that is used below to bound this quantity is similar to the one from the proof of Corollary 5 in~\cite{untangling}. 
    
    The \textit{level} of a region $R_i$, which we denote by $\ell(R_i)$, is defined as follows: If $R_i$ is fundamental, then its level is $0$. Otherwise, its level is obtained by adding $1$ to the largest of the levels of the two regions that were created when splitting $R_i$ using $C_i''$. The key observation here is that any two $R_i$'s with the same level are interior disjoint. This implies that the orders $m_i$ of the $T_i$'s such that $\ell(R_i)=l$ for a fixed $l$ add up to no more than \[O\left(n^4+|C'_0|+\sum_{i>0\text{ with } \ell(R_i)>l} |C''_i|\right).\] The largest level (which is $\ell(R_0)$) is clearly no more than $O(\log n)$, so one can check inductively that the total complexity of the $C_i''$'s with level $l$ is by $o(n^3)$, so the above expression is $O(n^4)$ . Also, every $R_i$ with level $l\geq1$ contains at least $\lceil\varepsilon^2 n\rceil(3/2)^{l-1}$ vertices of $G$, so there are at most $\varepsilon^{-2}(2/3)^{l-1}$ such regions. Now, the Cauchy-Schwarz inequality yields that the total complexity of the $C_i''$'s corresponding to regions in level $l$ is bounded from above by \[\sum_{i\text{ with } \ell(R_i)=l}4(L\sqrt{19m_i}+3n_i)\leq L'\frac{n^2}{\varepsilon}\left(\frac{2}{3}\right)^{\frac{l-1}{2}}+12n,\] where $L'$ is an absolute constant. Adding over all levels, we arrive at an upper bound of \[12n\cdot \ell(R_0)+ L'\frac{n^2}{\varepsilon}\sum_{l=1}^{\ell(R_0)}\sqrt{\frac{2}{3}}^{(l-1)}=O(n^2/\varepsilon),\] where we have used the fact that $\ell(R_0)=O(\log\varepsilon^{-1})=o(n)$. Note that the bound on the number of regions with level $l$ also yields that there are $O(1/\varepsilon^2)$ regions in total, not just at level $0$. Indeed, there are at most \[\sum_{l=1}^{\ell(R_0)}\varepsilon^{-2}(2/3)^{l-1}=O(1/\epsilon^2)\] regions with level $l\geq 1$. This observation will be used later.
    
    Next, we show how the fundamental regions can be further subdivided into smaller parts that satisfy the requirements in the statement of the theorem. Let us start by classifying the edges of $G$ as \textit{light} or \textit{heavy} depending on whether or not they are involved in less than $\varepsilon n^2/16$ crossings, respectively. By applying the crossing lemma to the subgraph that contains only the light edges, we get that the number of such edges is $O(\varepsilon^{1/2}n^2)$. A simple curve whose relative interior does not pass through any vertex or crossing point of $\mathcal{D}$ will be called \textit{clean}. Now, for any two points $p$ and $q$ on the plane, we define the distance $d_\mathcal{D}(p,q)$ as follows: 
    
    If $p$ and $q$ belong to the same fundamental region and at least one of them does not lie on its boundary, then $d_\mathcal{D}$ is the least positive integer $k$ such that there exists a clean curve which has endpoints $p$ and $q$, is contained in that region, and whose interior has at most $k$ intersection points with the edges of the drawing. Otherwise, set $d_\mathcal{D}(p,q)=\infty$. 
    
    For any point $p$ on the plane and any positive number $k$, let $B_{\mathcal{D}}(p,k)$ be the set that consists of those points $q$ such that $d_\mathcal{D}(p,q)\leq k$, and note that $B_{\mathcal{D}}(p,k)$ is closed and connected. We say that an edge of $G$ is \textit{encapsulated} by $B_{\mathcal{D}}(p,k)$ if both of its endpoints and all the points where it crosses another edge are contained in $B_{\mathcal{D}}(p,k)$, and at least one of these points belongs to the interior of this region.
    
    \begin{claim}
       
     For every point $p$ that does not belong to the boundary of a fundamental region, if $B_{\mathcal{D}}(p,\varepsilon n^2/4)$ is not the whole plane then it satisfies at least one of the properties below.
    
    \begin{enumerate}
        \item It contains at least $\varepsilon n^2/16$ intersections between the boundary of the fundamental region that contains $p$ and the edges of the drawing. 

        \item It contains one of the $C_i''$'s in its entirety.

        \item At least $\varepsilon^2n^4/32^2$ crossing points of $\mathcal{D}$ lie in its interior.

        \item It encapsulates at least $\varepsilon n^2/32$ light edges.
    \end{enumerate}

    \end{claim}

\begin{proof}
    
    First, we deal with the case where $B_{\mathcal{D}}(p,3\varepsilon n^2/16)$ contains a point from the boundary of the fundamental region $R_t$ that contains $p$. By starting from this point and walking along the corresponding connected component of the boundary of $R_t$ in both directions, we will either encounter at least $\varepsilon n^2/16$ intersections between the edges and the boundary of the region, all of which lie in $B_{\mathcal{D}}(p,\varepsilon n^2/4)$ (the fact that no vertex of $G$ is contained in the boundary of a fundamental region is important here), or will go around the entirety of one of the $C_i''$'s. Hence, at least one of the first two properties holds. From now on, we assume that $B_{\mathcal{D}}(p,3\varepsilon n^2/16)$ does not contain a point from the boundary of the corresponding fundamental region.
    
    Suppose now that $B_{\mathcal{D}}(p,\varepsilon n^2/16)$ contains a point from a heavy edge $e$. Then, we can find at least $\varepsilon n^2/16$ crossing points of $\mathcal{D}$ that involve $e$ and are contained in $B_{\mathcal{D}}(p,\varepsilon n^2/8)$. Apart from $e$, the edges involved in these crossings are all distinct due to the fact that $\mathcal{D}$ is simple. If any of these edges is light, then it must be encapsulated by $B_{\mathcal{D}}(p,3\varepsilon n^2/16)$, so we can assume that at least $\varepsilon n^2/32$ of them are heavy. Each of these heavy edges takes part in at least $\varepsilon n^2/16$ crossing that lie in $B_{\mathcal{D}}(p,3\varepsilon n^2/16)$. Summing over all of these edges and dividing by two to account for the fact that each crossing might be counted up to two times, we get that at least $(\varepsilon n^2/32)(\varepsilon n^2/16)/2$ crossing points of $\mathcal{D}$ belong to the interior of $B_{\mathcal{D}}(p,\varepsilon n^2/4)$.
    
    It remains to tackle the case where the interior of $B_{\mathcal{D}}(p,\varepsilon n^2/16)$ is disjoint from all the heavy edges. Given a clean curve $C$, let $c_1,c_2,\dots c_t$ denote the intersection points of its interior with the edges of $G$, and let $e_i$ be the edge that contains $c_i$; we define $w_G(C)$ as $\sum_{i=1}^t w_G(e_i)$. Consider a clean curve $C_p$ that connects $p$ to a point on the boundary of $B_{\mathcal{D}}(p,\varepsilon n^2/16)$ which is neither a crossing nor a vertex of $G$, and such that $w(C_p)$ is as small as possible amongst all curves of this kind. Furthermore, suppose that there is no curve with the aforementioned properties which has less intersection points with the edges of $G$ than $C_p$ does. Each of the at least $\varepsilon n^2/16$ points where the interior of $C_p$ intersects an edge must belong to a light edge. Since $\mathcal{D}$ is locally optimal, no two of these points belong to the same edge, as otherwise we would be able to construct a curve that contradicts the choice of $C_p$ by rerouting a portion of $C_p$ along such an edge. Hence, there are at least $\varepsilon n^2/16$ light edges that intersect the interior of $B_{\mathcal{D}}(p,\varepsilon n^2/16)$ and are thus encapsulated by $B_{\mathcal{D}}(p,\varepsilon n^2/4)$.
    \end{proof}

    We go back to the proof of Theorem~\ref{teo:regions}. The desired subdivision can now be obtained by means of a standard covering argument. Indeed, let $P=\{p_1,p_2,\dots,p_k\}$ be a set of points on the plane such that none of its elements represents a vertex or lies on an edge or on the boundary of a fundamental region, and which is maximal with the property that $B_{\mathcal{D}}(p_i,\varepsilon n^2/4)$ and $B_{\mathcal{D}}(p_j,\varepsilon n^2/4)$ are interior disjoint whenever $i\neq j$. Given that the total complexity of the boundaries of the fundamental regions is $O(n^2/\varepsilon)$, the number of $B_{\mathcal{D}}(p_i,\varepsilon n^2/4)$'s that satisfy property $1$ is $O(1/\varepsilon^2)$. Since the total number of $C_i''$'s is $O(1/\varepsilon^2)$ and there can be at most two regions which contain a specific $C_i''$ in its entirety, the same is true for those that satisfy property $2$. Because the $B_{\mathcal{D}}(p_i,\varepsilon n^2/4)$'s are interior disjoint, every crossing point of $\mathcal{D}$ belongs to the interior of at most one of them, and no edge can be encapsulated by more than one of them. Since $\mathcal{D}$ is simple, the number of crossing points is $O(n^4)$, so at most $O(1/\varepsilon^2)$ of the $B_{\mathcal{D}}(p_i,\varepsilon n^2/4)$'s satisfy property 3. Lastly, as there are no more than $O(\varepsilon^{1/2}n^2)$ light edges, the amount of $B_{\mathcal{D}}(p_i,\varepsilon n^2/4)$'s that encapsulate at least at $\varepsilon n^2/32$ light edges is upper bounded by $O(1/\varepsilon^2)$. This shows that $k=O(1/\varepsilon^2)$. 
    
    Next, notice that for every point $q$ we can find a $p_i$ such that such that $d_\mathcal{D}(p_i,q)\leq\varepsilon n^2/2$, or else $q$ could be added to $P$, contradicting its maximality. For each $q$, let $P(q)\subseteq P$ denote the set of $p_i$'s that minimize $d_\mathcal{D}(p_i,q)$, and let $p(q)$ be the element of $P(q)$ with the smallest subscript. Define $r_i$ as the closure of $\{q\ |\ p(q)=p_i\}$ for every $i\in\{1,\dots,k\}$. The $r_i$'s form a subdivision of the plane, and they are closed, connected and interior disjoint. Moreover, if $q\in r_i$ and $C_q$ is a clean curve that connects $q$ to $p_i$ such that its interior has precisely $d_\mathcal{D}(p_i,q)$ intersection points with the edges of $G$, then this curve lies completely within $r_i$. It follows that any two points in $r_i$ can be connected by a clean curve that goes through $p_i$ and whose interior crosses at most $\varepsilon n^2$ edges. Lastly, the $r_i$'s can once again be modified using a process similar to the one depicted in Figure~\ref{fig:1} so as to ensure that no vertex of $G$ lies on the boundary of any of them (this final step is where specific wording of property \textit{(III)} comes into play). The resulting regions satisfy properties \textit{(I), (II) and (III)}.

\end{proof}

\subsection{Cut distance and crossing numbers}

The proof of Theorem~\ref{teo:FoxPachSuk} by Fox, Pach and Suk relies on a regularity lemma for semi-algebraic graphs, also by the same authors~\cite{fox2014density}. In some sense, Theorem~\ref{teo:regions} will act as a purely combinatorial substitute of the said lemma (although this analogy should not be pushed too far). We will now use this theorem to show that any two graphs which are close with respect to $d_\square$ have similar crossing numbers.

\begin{theorem}\label{teo:closegraphs}
    Let $G_1$ and $G_2$ be edge-weighted $n$-vertex graphs on the same vertex set $V$ and write $d=d_\square(G_1,G_2)$. If $d\geq n^{-4}$, then \[|\mathop{\mathrm{cr}}(G_1)-\mathop{\mathrm{cr}}(G_2)|\leq Md^{1/4}n^4,\] where $M$ is an absolute constant. Moreover, $|\mathop{\mathrm{cr}}(G_1)-\mathop{\mathrm{cr}}(G_2)|\leq d^2n^8$ holds unconditionally.
\end{theorem}

We should that the content of this theorem lies in the first inequality.

\begin{proof}

    Suppose that $d\geq n^{-4}$. By adding edges of weight $0$, we can assume that every two elements of $V$ are adjacent in $G_1$. Consider a simple and locally optimal drawing $\mathcal{D}$ of $G_1$ and apply Theorem~\ref{teo:regions} with $\varepsilon=d^{1/8}$ to obtain a subdivision of the plane. Let $r_1,r_2,\dots,r_k$ be the regions that contain at least one vertex and, for each $i$, let $V_i\subset V$ denote the set of vertices that are contained in $r_i$. We have that $k\leq C /\varepsilon^2=C/d^{1/4}$ for some absolute constant $C$. Since $\varepsilon\geq n^{-1/2}$, each of the $V_i$'s has at most $\lceil\varepsilon^2 n\rceil<2\varepsilon^2 n$ elements. We say that an edge is \textit{long} if its endpoints belong to different $V_i$'s, and that it is \textit{short} otherwise.
    
    We shall use $\mathcal{D}$ as a blueprint for constructing a drawing $\mathcal{D}'$ of $G_2$. The elements of $V$ will be represented by the same points as in $\mathcal{D}$. Next, for every long edge $(u,v)\in E_{G_2}(V_i,V_j)$, choose a random edge $(u',v')\in E_{G_1}(V_i,V_j)$, where $(x,y)$ is chosen with probability $w_{G_1}(x,y)/e_{G_1}(V_i,V_j)$ (if $e_{G_1}(V_i,V_j)=0$, then $(u',v')$ is chosen uniformly at random from $E_{G_1}(V_i,V_j)$). These selections are carried out independently from each other for every long edge of $G_2$. Consider the points where the curve that represents $(u',v')$ in $\mathcal{D}$ crosses the union of the boundaries of $r_i$ and $r_j$ (if a section of the curve is contained in the union of the boundaries, we take only its endpoints) and label them as $p_1,p_2,\dots,p_t$ in the order that they appear as this curve is traversed from $u$ to $v$. If some $p_i$ lies on the intersection of the boundaries of $r_i$ and $r_j$, then we write $a(u,v)=b(u,v)=p_i$. Otherwise, we can choose two points $a(u,v)=p_i$ and $b(u,v)=p_{i\pm 1}$ such that $a(u,v)$ lies on the boundary of $r_i$ and $b(u,v)$ lies on the boundary of $r_j$. The curve that represents $(u,v)$ will be composed of three sections: one which goes from $u$ to $a(u,v)$, one that goes from $a(u,v)$ to $b(u,v)$ by following along the curve that represents $(u',v')$ in $\mathcal{D}$ (which is a single point in the case that $a(u,v)=b(u,v)$), and one last section from $b(u,v)$ to $v$. The first section will be constructed so that it is fully contained in $r_i$ and it has as few crossings as possible with the edges of $G_1$ that have no endpoint in $V_i$. Analogously, we will draw the last section so that it lies within $r_j$ and has the least possible number of crossings with the edges of $G_1$ that have no endpoint in $V_j$. The first and last sections of every curve that represents a long edge of $G_2$ will be called \textit{extremal}. See Figure~\ref{fig:3}. Each short edge with both endpoints in some $V_i$ will be represented by a curve that is contained in $r_i$ and has the least possible number of crossings with the edges of $G_1$ that have no endpoint in $V_i$. 

    \begin{figure}[!htbp]
\centering
\includegraphics[scale=0.7]{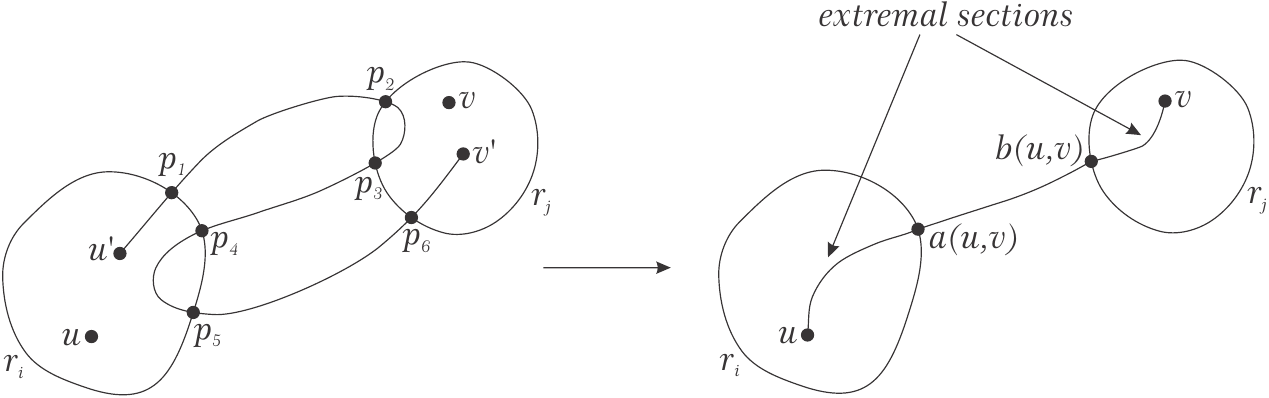}
\caption{On the right, we have the curve that represents $(u',v')$ in $\mathcal{D}$. Setting $a(u,v)=p_4$ and $b(u,v)=p_3$ and proceeding as described in the proof, we arrive at a curve that connects $u$ and $v$. Note that we could have also chosen $a(u,v)=p_1$ and $b(u,v)=p_2$, or $a(u,v)=p_5$ and $b(u,v)=p_6$.}
\label{fig:3}
\end{figure}
    
    Let $E_i$ denote the set of curves which consists of the short edges of $G_2$ with endpoints in $V_i$, and the extremal sections with an endpoint in $V_i$. By property \textit{(III)} in the statement of Theorem~\ref{teo:regions}, the above construction can be carried out so that each element of $E_i$ has at most $\varepsilon n^2$ crossings with the edges of $G_1$ that have no endpoint in $V_i$, and such that any two elements of $E_i$ have a finite number of points in common. There might still be some pairs of elements of $E_i$ which cross more than once, but these multiple crossings can be eliminated by rerouting some of the curves without increasing the total number of crossing between the elements of $E_i$ and the edges of $G_1$ with no endpoint in $V_i$. Finally, observe that some pairs of long edges of $E_{G_2}(V_i,V_j)$ may have infinitely many points in common if their non-extremal sections coincide. Thus, in order to obtain an actual graph drawing, we might have to perturb these edges slightly, as shown in Figure~\ref{fig:4}. We point out that the resulting drawing $\mathcal D'$ is not necessarily simple. 

\begin{figure}[!htbp]
\centering
\includegraphics[scale=0.7]{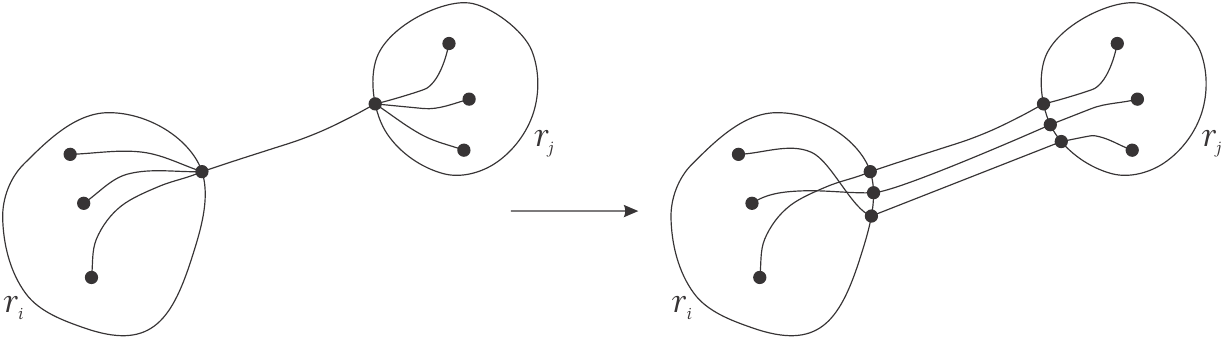}
\caption{Edges that share entire non-extremal sections can be modified as shown above. It is important that these modifications be carried out so that any two extremal sections cross at most once and the number of crossings between the extremal sections in $r_i$ and the edges of $G_1$ that do not have an endpoint in $r_i$ does not increase (the same goes for $r_j$).}
\label{fig:4}
\end{figure}
    
    The goal now is to bound the expected value of $\mathop{\mathrm{cr}}(G_2,\mathcal{D}')$. A long edge of $G_2$ will be called \textit{lonely} if it belongs to $E_{G_2}(V_i,V_j)$ and $e_{G_1}(V_i,V_j)\leq dn^2$. Note that $G_2$ has no more than $dk^2n^2\leq C^2d^{1/2}n^2$ lonely edges. Delete all the lonely edges from $\mathcal{D}'$; we will later describe how they can be reinserted without creating too many crossings. 
    There are four types of crossings in $\mathcal{D}'$: the ones that involve two non-extremal sections (\textit{type} $1$), those which involve a non-extremal section and the interior of an extremal one (\textit{type} $2$), those that involve a short edge and a non-extremal section (\textit{type} $3$), and finally the ones that involve only short edges and extremal sections (\textit{type} 4). 
    
    Every crossing of type $1$ in $\mathcal{D}'$ can be traced back to a crossing in $\mathcal{D}$. If we think about it in this way, then a crossing in $\mathcal{D}$ between an edge $(u,v)\in E_{G_1}(V_i,V_j)$ and an edge $(u',v')\in E_{G_1}(V_s,V_t)$ contributes at most \[\frac{w_{G_1}(u,v)\cdot w_{G_1}(u',v')\cdot e_{G_2}(V_i,V_j)\cdot e_{G_2}(V_s,V_t)}{e_{G_1}(V_i,V_j)\cdot e_{G_1}(V_s,V_t)}\] to $\mathbb{E}[\mathop{\mathrm{cr}}(G_2,\mathcal{D}')]$. Let $c_\mathcal{D}(i,j|s,t)$ denote the total weighted sum of the crossings between an edge of $E_{G_1}(V_i,V_j)$ and an edge from $E_{G_1}(V_s,V_t)$ in $\mathcal{D}$. Then, the total contribution of these crossings to $\mathbb{E}[\mathop{\mathrm{cr}}(G_2,\mathcal{D}')]$ is upper bounded by \[c_\mathcal{D}(i,j|s,t)\frac{e_{G_2}(V_i,V_j)\cdot e_{G_2}(V_s,V_t)}{e_{G_1}(V_i,V_j)\cdot e_{G_1}(V_s,V_t)}\leq c_\mathcal{D}(i,j|s,t)\frac{(e_{G_1}(V_i,V_j)+d n^2)(e_{G_1}(V_s,V_t)+d n^2)}{e_{G_1}(V_i,V_j)\cdot e_{G_1}(V_s,V_t)}.\] Adding over all $4$-tuples $(i,j,s,t)$ of numbers in $\{1,\dots,k\}$ (two tuples are considered the same if the correspond to the same pairs of parts of $\mathcal{P}$), we get that the expected weighted sum of the crossings of type 1 in $\mathcal{D}'$ does not exceed \[\sum_{(i,j,s,t)}c_\mathcal{D}(i,j|s,t)\left(1+\frac{e_{G_1}(V_i,V_j)d n^2+e_{G_1}(V_s,V_t)d n^2+d^2n^4}{e_{G_1}(V_i,V_j)\cdot e_{G_1}(V_s,V_t)}\right)\] \[\leq\sum_{(i,j,s,t)}c_\mathcal{D}(i,j|s,t)+e_{G_1}(V_i,V_j)d n^2+e_{G_1}(V_s,V_t)d n^2+d^2n^4\] \[\leq\mathop{\mathrm{cr}}(G_1,\mathcal{D})+k^4(\varepsilon^2 n)^2 d n^2+k^4d^2n^4\leq\mathop{\mathrm{cr}}(G_1,\mathcal{D})+C^4d^{1/2}n^4+C^4d^2n^4\leq\mathop{\mathrm{cr}}(G_1,\mathcal{D})+2C^4d^{1/2}n^4.\]

    Next, we bound the amount of crossings of types $2$ and $3$ that are contained in $r_i$. Note that each of them can be traced back to an intersection point between a long edge of $G_1$ that has no endpoint in $V_i$ and a either a short edge of $G_2$ with endpoints in $V_i$ or an extremal section with endpoint in $V_i$ (in other words, an intersection point between an edge of $G_1$ with no endpoint in $V_i$ and an element of $E_i$). By construction, and thanks to property \textit{(III)} in the statement of Theorem~\ref{teo:regions}, we have  that on average each of the at most $\lceil\varepsilon^2 n\rceil n$ elements of $E_i$ takes part in no more than $\varepsilon n^2$ crossings with the long edges of $G_1$ that have no endpoint in $V_i$. On top of this, if the long edge of $G_1$ involved in one of these crossings belongs to $E_{G_1}(V_s,V_t)$, then the crossing contributes at most \[\frac{e_{G_2}(V_s,V_t)}{e_{G_1}(V_s,V_t)}\leq\frac{e_{G_1}(V_s,V_t)+dn^2}{e_{G_1}(V_s,V_t)}\] to $\mathbb{E}[\mathop{\mathrm{cr}}(G_2,\mathcal{D}')]$. Since the lonely edges were deleted, we may assume that $e_{G_1}(V_s,V_t)> dn^2$, and thus the above quantity is less than $2$. Whence, the expected value of the weighted sum of the crossings of types $2$ and $3$ is upper bounded by $2\lceil\varepsilon^2n\rceil n\cdot\varepsilon n^2<4\varepsilon^3n^4=4d^{3/8}n^4$.

    Lastly, as any two short edges or extremal sections that are fully contained in $r_i$ induce at most $1$ crossing, the number of crossings of type $4$ does not surpass \[k\binom{\lceil\varepsilon^2n\rceil}{2}\binom{n}{2}\leq k\varepsilon^4n^4\leq Cd^{1/4}n^4.\]

    Now, we add the at most $C^2d^{1/2}n^2$ lonely edges back into the drawing. In order to do this, we first transform $\mathcal{D}'$ into a simple and locally optimal drawing by applying a sequence of operations that does not increase the weighted sum of the crossings. At this point, each lonely edge can be represented by a curve that does not cross any non-lonely edge more than once. We can get rid of multiple crossings between the lonely edges by rerouting some of the curves without increasing the number of crossings between lonely and non-lonely edges. At the end, we get a drawing $\mathcal{D}''$ of $G_2$ that satisfies \[\mathbb{E}[\mathop{\mathrm{cr}}(G_2,\mathcal{D}'']\leq\mathbb{E}[\mathop{\mathrm{cr}}(G_2,\mathcal{D}')]+C^2d^{1/2}n^2\binom{n}{2}\leq\mathbb{E}[\mathop{\mathrm{cr}}(G_2,\mathcal{D}')]+C^2d^{1/2}n^4\] \[\leq\mathop{\mathrm{cr}}(G_1,\mathcal{D})+2C^4d^{1/2}n^4+4d^{3/8}n^4+Cd^{1/4}n^4+C^2d^{1/2}n^4\] \[\leq\mathop{\mathrm{cr}}(G_1,\mathcal{D})+Md^{1/4}n^4,\] where $M$ is an absolute constant. As this can be repeated for every drawing $\mathcal{D}$ of $G_1$, we arrive at $\mathop{\mathrm{cr}}(G_2)\leq\mathop{\mathrm{cr}}(G_1)+Md^{1/4}n^4$. Analogously, $\mathop{\mathrm{cr}}(G_2)\leq\mathop{\mathrm{cr}}(G_1)+Md^{1/4}n^4$. This proves the first part of the statement. For the second part, a direct computation yields $\mathop{\mathrm{cr}}(G_2,\mathcal{D})\leq\mathop{\mathrm{cr}}(G_1,\mathcal{D})+d^2n^8$ (that is, we do not need to modify the drawing of $G_1$ at all to get the desired inequality), and we can finish as above .
\end{proof}

Next, we address how taking the blow-up of a graph affects the crossing number.

\begin{theorem}\label{teo:blowup}
For any edge-weighted $n$-vertex graph $G=(V,E)$ and any positive integer $m$, we have that \[0\leq\mathop{\mathrm{cr}}(G[m])-m^4\mathop{\mathrm{cr}}(G)\leq n^3m^4.\]
\end{theorem}

\noindent\textit{Remark.} This was shown to be true for the rectilinear crossing number (instead of the crossing number) in~\cite{approximatingrect}; the proofs are very similar, and the fact that this holds also for the crossing number was already known to Fox, Pach and Suk.

\begin{proof}
We begin by proving the inequality on the left. Consider a simple drawing $\mathcal{D}$ of $G[m]$ which attains $\mathop{\mathrm{cr}}(G)$. From this drawing, we obtain a drawing $\mathcal{D}'$ of $G$ in the following way: Independently for each $v\in V$, choose uniformly at random an element $v'$ from $U_v$ (see the definition of $G[m]$ in Section~\ref{sec:cutdistance}); the point $p_v$ that represents $v'$ in $\mathcal{D}$ will be used to represent $v$ in $\mathcal{D}'$. Once all these points have been chosen, notice that for every two adjacent vertices $u$ and $v$ of $G$ the points $p_v$ and $p_u$ are connected by a curve in $\mathcal{D}$; this is the curve that will represent the edge $(u,v)$ in $\mathcal{D}'$. Clearly, $\mathcal{D}'$ is simple. For any two non-adjacent edges $e_1=(a,b),e_2=(c,d)\in E$, the probability that they cross in $\mathcal{D}'$ is $ c(U_a,U_b|U_c,U_d)/m^4$, where $c(U_a,U_b|U_c,U_d)$ denotes the number of quadruples $(a',b',c',d')$ with $a'\in U_a$, $b'\in U_b$, $c'\in U_c$, $d'\in U_d$ such that the edges $(a',b')$ and $(c',d')$ cross each other in $\mathcal{D}$. From this, we get that \[\mathbb{E}[\mathop{\mathrm{cr}}(G,\mathcal{D}')]\leq\mathop{\mathrm{cr}}(G[m],\mathcal{D})/m^4=\mathop{\mathrm{cr}}(G[m])/m^4.\] Since $\mathop{\mathrm{cr}}(G)\leq\mathbb{E}[\mathop{\mathrm{cr}}(G,\mathcal{D}')]$, the result follows.

Now we prove the inequality on the right. Take a drawing $\mathcal{D}$ of $G$ that attains $\mathop{\mathrm{cr}}(G)$ and let $\varepsilon$ be a very small positive real (this will be made more precise in a moment). We obtain a drawing $\mathcal{D}'$ of $G[m]$ as follows: For every $v\in V$, replace the point that represents $v$ by a set of $m$ distinct points, each of them at distance at most $\varepsilon$ from the original one; these points will represent the elements of $U_v$. For each $(u,v)\in E$, we draw all the edges between $U_u$ and $U_v$ by following along the curve that represents $(u,v)$ in $\mathcal{D}$ (this makes sense if $\varepsilon$ is small enough), making sure that no two edges of $G[m]$ cross each other more than once. We call a $4$-tuple $(a,b,c,d)$ of vertices of $G[m]$ \textit{bad} if two of its elements belong to the same $U_v$, and \textit{good} otherwise. If $\varepsilon$ is small enough, then for any good $4$-tuple $(a,b,c,d)$ with $a\in U_a'$, $b\in U_b'$, $c\in U_c'$, $d\in U_d'$ the edges $(a,b)$ and $(c,d)$ cross in $\mathcal{D}'$ if and only if $(a,b)$ and $(c,d)$ cross in $\mathcal{D}$. On the other hand, the number of crossings coming from bad $4$-tuples is less than $n\binom{m}{2}\binom{nm}{2}$, so we have that \[\mathop{\mathrm{cr}}(G[m])\leq\mathop{\mathrm{cr}}(G[m],\mathcal{D}')\leq m^4\mathop{\mathrm{cr}}(G,\mathcal{D})+n\binom{m}{2}\binom{nm}{2}\leq m^4\mathop{\mathrm{cr}}(G)+n^3m^4.\]  
\end{proof}

By theorems~\ref{teo:closegraphs} and~\ref{teo:blowup}, the crossing density of graphs satisfies properties (i) and (ii) in Section~\ref{sec:estimable}, and it clearly satisfies (iii) as well. This implies that the crossing density is an estimable graph parameter. That the rectilinear crossing density is estimable was already shown in~\cite{approximatingrect}.

\begin{theorem}\label{teo:estimable}
    The crossing and rectilinear crossing densities are estimable graph parameters. As a consequence, there exist two functionals $g$ and $\overline{g}$ on $W_0$ such that $\lim_{n\rightarrow\infty}\mathop{\mathrm{cd}}(G_n)=g(W)$ and $\lim_{n\rightarrow\infty}\overline{\mathop{\mathrm{cd}}}(G_n)=\overline{g}(W)$ whenever $G_1,G_2,\dots$ converges to $W$ and the number of vertices of $G_n$ goes to infinity with $n$.
\end{theorem}

\subsection{The algorithm}

Here, we present the algorithm that was promised in Theorem~\ref{teo:main}. The first three steps of the algorithm are completely standard and, in particular, very similar to the algorithm from~\cite{approximatingrect} (see Theorem~\ref{teo:FoxPachSuk}).

\begin{enumerate}
    \item We are given an $n$-vertex graph $G=(V,E)$ as input.

    \item Set $\varepsilon=(\log\log n)^{-\frac{1}{2c}}$ and use Theorem~\ref{teo:FriezeKannan} to find an equitable Frieze-Kannan $\varepsilon$-regular partition $\mathcal{P}=\{V_1,V_2,\dots,V_k\}$ of $G$, where $k\leq O(2^{\sqrt{\log\log n}})$ (here, $c$ denotes the constant in the satement of the said theorem). This takes $n^{2+o(1)}$ time.

    \item Consider the edge-weighted graph $G/\mathcal{P}$ and compute its crossing number using Theorem~\ref{teo:exactcr}. Then, output $\text{cr}(G/\mathcal{P})\cdot (\lfloor n/k\rfloor)^4$. This can be done in $2^{O(k^4\log k)}=n^{o(1)}$ time.

\begin{proof}[Proof of Correctness]
    Let $H$ denote the blow-up graph $(G/\mathcal{P})[\lfloor n/k\rfloor]$. By Theorem~\ref{teo:blowup}, the crossing number of $H$ satisfies:\[\text{cr}(G/\mathcal{P})(\lfloor n/k\rfloor)^4\leq\text{cr}(H)\leq\text{cr}(G/\mathcal{P})(\lfloor n/k\rfloor)^4+n^4/k.\] The graph $H$ can be obtained from $G_\mathcal{P}$ by removing less than $k$ vertices (along with the attached edges), thus \[\text{cr}(H)\leq \text{cr}(G_\mathcal{P})\leq \text{cr}(H)+kn^3.\] Finally, $\mathcal{P}$ was chosen so that $d_\square(G,G_\mathcal{P})\leq\varepsilon$, so Theorem~\ref{teo:closegraphs} yields $|\text{cr}(G)-\text{cr}(G_\mathcal{P})|\leq M\varepsilon^{1/4}n^4$. Together, these inequalities imply that \[|\text{cr}(G/\mathcal{P})(\lfloor n/k\rfloor)^4- \text{cr}(G)|\leq n^4/k+kn^3+M\varepsilon^{1/4}n^4=O\left(\frac{n^4}{(\log\log n)^{\frac{1}{8c}}}\right).\]
\end{proof}

    \item Write $\delta'=1/8c$. If we also wish to find a drawing of $G$ with $\mathop{\mathrm{cr}}(G)+O(n^4/(\log\log n)^{\delta'})$ crossings, then we first need access to a simple and locally optimal drawing of $G_\mathcal{P}$ (or of a graph that is very close to $G_\mathcal{P}$ with respect to $d_\square$) with crossing number $\mathop{\mathrm{cr}}(G_\mathcal{P})+O(n^4/(\log\log n)^{\delta'})$. A drawing of $G_\mathcal{P}$ with this many crossings can be obtained by combining the drawing of $G/\mathcal{P}$ provided by~\ref{teo:exactcr} with the technique used in the proof of the right hand side inequality in the statement of Theorem~\ref{teo:blowup}. Next, we round the weights of the edges of $G_{\mathcal P}$ so that they can be written as an integer divided by some fixed integer $s=\Theta(\log\log n)^{\delta'}$, thus obtaining a new graph $G_\mathcal{P}'$ whose crossing number differs from that of $G_\mathcal{P}$ by at most $O(n^4/(\log\log n)^{\delta'})$. In order to construct a locally optimal drawing of $G_\mathcal{P}'$, we iteratively refine the drawing by using Dijkstra's algorithm to redraw the edges until it is no longer possible to reduce the crossing number of the drawing in this manner. Each step reduces the weighted sum of the crossings by at least $1/s^2$, so the whole process takes polynomial time. Now, we can follow the proof of Theorem~\ref{teo:closegraphs}, which is essentially algorithmic, to obtain a random drawing of $G$. At some point during this procedure, we also require a locally optimal drawing of a certain subgraph of $G$; this drawing can be constructed just as we did for $G_\mathcal{P}'$, since every step will reduce the weighted sum of the crossings by at least $1$. The expected number of crossings in the resulting drawing of $G$ is $\mathop{\mathrm{cr}}(G)+O(n^4/(\log\log n)^{\delta'})$, and the probability that the algorithm fails to deliver a drawing with crossing number close to this quantity can be bounded from above using Markov's inequality together with the fact that no drawing with crossing number less than $\mathop{\mathrm{cr}}(G)$ can ever be generated.

    \textit{Remark.} The only step in part $4$ of the above algorithm at which we produce an actual graph drawing is at the end. Before that, drawings are processed not as geometric objects, but combinatorial ones (in the sense that they are stored as planar maps with a vertex for each crossing point and for each vertex of the original graph). The final drawing can be constructed using any efficient algorithm for drawing planar graphs (see Chapter 6 in~\cite{tamassia_handbook} for many such algorithms).
\end{enumerate}

\section{Crossing densities of graphons}\label{sec:graphons}

\subsection{Defining the crossing and rectilinear crossing densities of graphons}\label{sec:densities} Our first goal in this section will be to introduce the \textit{rectilinear crossing density} of a graphon.  Let us start with some preliminary definitions:

For any four (not necessarily distinct) points $x,y,z,w\in\mathbb{R}^2$, let $I(x,y,z,w)=1$ if $\text{conv}\{x,y\}\cap\text{conv}\{z,w\}\neq\emptyset$ and $I(x,y,z,w)=0$ otherwise. We claim that $I$, as a function from $(\mathbb{R}^2)^4$ to $\mathbb{R}$, is measurable. This follows from the fact that $I(-1)$ can be thought of as an algebraic subset of $\mathbb R^8$. Alternatively, we can let \[X=\{(x,y,z,w)\in(\mathbb{R}^2)^4 :\ |\{x,y,z,w\}|\leq 3\ \text{or}\ x,y,z,w\ \text{are contained in a line}\}.\] The set $X$ is measurable, and the restriction of $I$ to $X$ can easily be seen to be measurable too. Hence, it suffices to show that $I^{-1}(1)\backslash X\subset(\mathbb{R}^2)^4$ is measurable. For every positive integer $m$, consider a grid-like subdivision of the plane into interior disjoint squares of side length $1/m$, which we denote by $S_{m,1},S_{m,2},\dots$ in any order. Let $A_m\subset\mathbb{N}^4$ be the set of all $4$-tuples of distinct indices $(a,b,c,d)$ such that $I(x,y,z,w)$ must equal $1$ whenever $x\in S_{m,a},\ y\in S_{m,b},\ z\in S_{m,c},\ w\in S_{m,d}$. Write \[B_m=\bigcup_{(a,b,c,d)\in A_m}S_{m,a}\times S_{m,b}\times S_{m,c}\times S_{m,d}.\] By definition, $I(B_m)=1$, and it is not hard to see that the measurable set $B=\cup_{i=1}^\infty B_i$ covers $I^{-1}(1)\backslash X$. Thus, $I^{-1}(1)\backslash X=B\backslash X$, which implies the claim.

As a consequence of the above result, if $f:[0,1]\rightarrow\mathbb{R}^2$ is measurable then the function $I_f:[0,1]^4\rightarrow[0,1]$ defined by $I_f(x,y,z,w)=I(f(x),f(y),f(z),f(w))$ is measurable too. Denote the family of all bounded measurable functions from $[0,1]$ to $\mathbb{R}^2$ by $\mathcal{F}$. Given a graphon $W$ and a function $f\in\mathcal{F}$, let \[\overline{\mathop{\mathrm{cd}}}(W,f)=\frac{1}{8}\int_{[0,1]^4}W(x_1,x_2)W(x_3,x_4)I_f(x_1,x_2,x_3,x_4)\prod_{i=1}^4 dx_i.\]
By the discussion above, this integral is well defined. The function $f$ can be thought of as a straight-line drawing of $W$; in fact, the definition of $\overline{\mathop{\mathrm{cd}}}(W,f)$ is essentially the continuous (and normalized) analog of the definition of $\overline{\mathop{\mathrm{cr}}}(G,\mathcal{D})$ given in Section~\ref{sec:prelim}. Note that we divide by $8$ to make up for the fact the each crossing is counted eight times by the integral. The \textit{rectilinear crossing density} of $W$ is now defined as \[\overline{\mathop{\mathrm{cd}}}(W)=\inf_{f\in\mathcal{F}}\overline{\mathop{\mathrm{cd}}}(W,f).\]

The definition of the \textit{crossing density} of a graphon is somewhat more intricate. Let $\mathcal{C}$ denote the family of all curves $c:[0,1]\rightarrow\mathbb{R}^2$ that are either simple or constant\footnote{Up to this point, we had used the term curve to refer to the image of a continuous function from an interval to the plane. From here on out, both the function and its image will be referred to as curves.} A \textit{simple graphon drawing} is a function $\mathcal{D}:[0,1]^2\rightarrow\mathcal{C}$ such that the following properties hold:

\begin{enumerate}
    \item For any $x,y,t\in[0,1]$, $\mathcal{D}(x,y)(t)=\mathcal{D}(y,x)(1-t)$.

    \item For every $x\in[0,1]$, the image of $\mathcal{D}(x,x)$ is a single point, which we denote by $\mathcal{D}(x)$.

    \item For any $x,y\in[0,1]$, $\mathcal{D}(x,y)(0)=\mathcal{D}(x)$ and $\mathcal D(x,y)(1)=\mathcal D(y)$.

    \item For any $x_1,x_2,y_1,y_2\in[0,1]$, if the images of $\mathcal{D}(x_1,y_1)$ and $\mathcal{D}(x_2,y_2)$ have nonempty intersection (as subsets of $\mathbb{R}^2$), then their intersection is connected.
\end{enumerate}

For a simple graphon drawing $\mathcal{D}$ and $x,y,z,w\in[0,1]$, let $I_\mathcal{D}(x,y,z,w)=1$ if the curves $\mathcal{D}(x,y)$ and $\mathcal{D}(z,w)$ have at least one point in common, and $I_\mathcal{D}(x,y,z,w)=0$ otherwise. Let $\mathcal{D}_\mathcal{W}$ denote the family of all simple graphon drawings $\mathcal{D}$ such that $I_\mathcal{D}$ is measurable. Given a graphon $W$ and a drawing $\mathcal{D}\in\mathcal{D}_\mathcal{W}$, write \[\mathop{\mathrm{cd}}(W,\mathcal{D})=\frac{1}{8}\int_{[0,1]^4}W(x_1,x_2)W(x_3,x_4)I_\mathcal{D}(x_1,x_2,x_3,x_4)\prod_{i=1}^4 dx_i.\] The \textit{crossing density} of $W$ is defined as \[\mathop{\mathrm{cd}}(W)=\inf_{\mathcal{D}\in\mathcal{D}_\mathcal{W}}\mathop{\mathrm{cd}}(W,\mathcal{D}).\] 

\noindent\textit{Remark.} We feel the need to clarify that, even though we make a slight abuse of notation by using each of $\mathop{\mathrm{cd}}$ and $\overline{\mathop{\mathrm{cd}}}$ to denote both a graph parameter and a functional on $\mathcal{W}_0$, it is not necessarily true that $\mathop{\mathrm{cd}}(G)=\mathop{\mathrm{cd}}(W_G)$ or $\overline{\mathop{\mathrm{cd}}}(G)=\overline{\mathop{\mathrm{cd}}}(W_G)$ (see Theorem~\ref{teo:G_W}). Also, while it is not completely evident from the definitions that $\overline{\mathop{\mathrm{cd}}}(W)\leq\mathop{\mathrm{cd}}(W)$, this inequality does hold for every $W\in\mathcal{W}_0$.

It is natural to ask whether for every graphon $W$ there is an $f\in\mathcal{F}$ such that $\overline{\mathop{\mathrm{cr}}}(W)=\overline{\mathop{\mathrm{cr}}}(W,f)$, or a $\mathcal{D}\in\mathcal{D}_\mathcal{W}$ such that $\mathop{\mathrm{cd}}(W)=\mathop{\mathrm{cd}}(W,\mathcal{D})$ (in other words, can the infimum be substituted by a minimum in the above definitions?). Unfortunately, we do not know the answer to either of these question. The following result, however, is an immediate consequence of the Banach-Alaoglu theorem.

\begin{lemma}\label{teo:weakclosure}
There exist two families $\overline{\mathcal{I}}$ and $\mathcal{I}$ of measurable functions from $\mathbb{R}^4$ to $\mathbb{R}$ such that, for every $W\in\mathcal{W}_0$, we have \[\overline{\mathop{\mathrm{cd}}}(W)=\min_{I\in\overline{\mathcal{I}}}\frac{1}{8}\int_{[0,1]^4}W(x_1,x_2)W(x_3,x_4)I(x_1,x_2,x_2,x_4)\prod_{i=1}^4 dx_i\] and \[\mathop{\mathrm{cd}}(W)=\min_{I\in\mathcal{I}}\frac{1}{8}\int_{[0,1]^4}W(x_1,x_2)W(x_3,x_4)I(x_1,x_2,x_2,x_4)\prod_{i=1}^4 dx_i.\] Furthermore, these families can be chosen so that all of their elements take values in $[0,1]$.
\end{lemma}

\begin{proof}
    Consider the sets $A=\{I_f\ |\ f\in\mathcal{F}\}$ and $B=\{I_\mathcal{D}\ |\ \mathcal{D}\in\mathcal{D}_\mathcal{W}\}$. By Banach-Alaoglu, any sequence of elements of $A$ has a subsequence that is convergent in the weak* topology\footnote{A sequence $f_1,f_2,\dots$ of integrable functions defined on $[0,1]^4$ is weak* convergent if for any integrable function $g$ defined on the same set we have that $\langle f_n,g\rangle=\int_{[0,1]^4}f_ng\text{ } dx$ converges as $n$ goes to infinity.}, and the same is true for $B$. Now, let $W\in\mathcal{W}_0$ and take a sequence $f_1,f_2,\dots$ of elements of $\mathcal{F}$ such that $\lim_{n\rightarrow\infty}\overline{\mathop{\mathrm{cr}}}(W,f_n)=\overline{\mathop{\mathrm{cr}}}(W)$. We can pass to a subsequence $f'_1,f ' _2,\dots$ such that $I_{f'_1},I_{f'_2},\dots$ converges in the weak* topology to a measurable function $I$, which can clearly be chosen so that it takes values in $[0,1]$. $I$ satisfies \[\overline{\mathop{\mathrm{cd}}}(W)=\frac{1}{8}\int_{[0,1]^4}W(x_1,x_2)W(x_3,x_4)I(x_1,x_2,x_2,x_4)\prod_{i=1}^4 dx_i.\] Hence, setting $\overline{\mathcal{I}}$ to be the set formed by all elements of $A$ and all sequential limit points of $A$ with respect to the weak* topology which have image in $[0,1]$ does the job. In the same way, we can show that the set $\mathcal{I}$ which consists of all elements of $B$ and all of its sequential limit points with image in $[0,1]$ has the required properties.
\end{proof}

For every $W\in\mathcal{W}_0$, let $\overline{I}_W\in\overline{\mathcal{I}}$ and $I_W\in\mathcal{I}$ denote two functions that attain the minimums in the statement of Lemma~\ref{teo:weakclosure}. The questions above can now be restated as follows. Given a graphon $W$, can the functions $\overline{I}_W$ and $I_W$ always be chosen so that $\overline{I}_W=I_f$ for some $f\in\mathcal{F}$ and $I_W=I_\mathcal{D}$ for some $\mathcal{D}\in\mathcal{D}_\mathcal{W}$? We will come back to this in Section~\ref{sec:final}.

\subsection{Continuity with respect to \texorpdfstring{$d_\square$}{TEXT} and connection to graphs}

Here, we prove that the crossing and rectilinear crossing densities are continuous with respect to $d_\square$, and that they extend the corresponding graph parameters in a precise sense (see Theorem~\ref{teo:continuous}). We must first establish a weaker sort of continuity for $\text{cd}$ and $\overline{\text{cd}}$.

\begin{lemma}\label{teo:pointwise}
    If $W_1,W_2,\dots$ is a sequence of graphons that converges to another graphon $W$ almost everywhere, then \[\lim_{n\rightarrow\infty}\overline{\mathop{\mathrm{cd}}}(W_n)=\overline{\mathop{\mathrm{cd}}}(W)\text{ and } \lim_{n\rightarrow\infty}\mathop{\mathrm{cd}}(W_n)=\mathop{\mathrm{cd}}(W).\]
\end{lemma}

\begin{proof}
 The \textit{tensor product} of two graphons $W_1$ and $W_2$ is the function $W_1\otimes W_2:[0,1]^4\rightarrow[0,1]$ defined by $(W_1\otimes W_2)(x,y,z,w)=W_1(x,y)W_2(z,w)$. 
 
 Consider a sequence of graphons $W_1,W_2,\dots$ that converges to $W\in\mathcal{W}_0$ almost everywhere. Clearly, the sequence $W_1\otimes W_1,W_2\otimes W_2,\dots$ converges to $W\otimes W$ almost everywhere. Choose a subsequence $W_1',W_2',\dots$ of $W_1,W_2,\dots$ with $\lim_{n\rightarrow\infty}\overline{\mathop{\mathrm{cd}}}(W_n')=\liminf_{n\rightarrow\infty}\overline{\mathop{\mathrm{cd}}}(W_n)$ and, for every $n$, let $f_n$ be an element of $\mathcal{F}$ such that $\overline{\mathop{\mathrm{cd}}}(W_n,f_n)\leq\overline{\mathop{\mathrm{cd}}}(W_n)+1/2^n$. By Banach-Alaoglu, we can pass to a further subsequence $W_{s(1)}',W_{s(2)}',\dots$ of $W_1',W_2',\dots$ such that $\overline{I}_{f_{s(1)}},\overline{I}_{f_{s(2)}},\dots$ converges to some $I\in\overline{\mathcal{I}}$ in the weak* topology. This way, \[\overline{\mathop{\mathrm{cd}}}(W_{s(n)}')\leq
 \frac{1}{8}\int_{[0,1]^4} (W_{s(n)}'\otimes W_{s(n)}')\overline{I}_{f_{s(n)}}(x_1,x_2,x_3,x_4)\prod_{i=1}^4 dx_i\leq\overline{\mathop{\mathrm{cd}}}(W_{s(n)}')+1/2^{s(n)}\] for every positive integer $n$, and so the expression in the middle converges to \[\lim_{n\rightarrow\infty}\overline{\mathop{\mathrm{cd}}}(W_{s(n)}')=\lim_{n\rightarrow\infty}\overline{\mathop{\mathrm{cd}}}(W_n')=\liminf_{n\rightarrow\infty}\overline{\mathop{\mathrm{cd}}}(W_n)\] as $n$ goes to infinity. 
 
 On the other hand, a standard  argument from analysis shows that the sequence $(W_{s(1)}\otimes W_{s(1)})\cdot\overline{I}_{f_{s(1)}},(W_{s(1)}\otimes W_{s(1)})\cdot\overline{I}_{f_{s(1)}},\dots$ converges to $(W\otimes W)\cdot I$ in the weak* topology. Putting everything together, we get that \[\liminf_{n\rightarrow\infty}\overline{\mathop{\mathrm{cd}}}(W_n)=\lim_{n\rightarrow\infty}\overline{\mathop{\mathrm{cd}}}(W_{s(n)}')=\overline{\mathop{\mathrm{cd}}}(W,I)\geq\overline{\mathop{\mathrm{cd}}}(W).\] 

 Now, observe that \[\overline{\mathop{\mathrm{cd}}}(W_n)\leq\frac{1}{8}\int_{[0,1]^4} (W_n\otimes W_n)\overline{I}_W(x_1,x_2,x_3,x_4)\prod_{i=1}^4 dx_i\] for every $n$. Since the expression on the right converges to $\overline{\mathop{\mathrm{cd}}}(W)$ as $n$ goes to infinity, we arrive at $\limsup_{n\rightarrow\infty}\overline{\mathop{\mathrm{cd}}}(W_n)\leq\overline{\mathop{\mathrm{cd}}}(W)$, but we also had $\liminf_{n\rightarrow\infty}\overline{\mathop{\mathrm{cd}}}(W_n)\geq\overline{\mathop{\mathrm{cd}}}(W)$; this can only occur if $\lim_{n\rightarrow\infty}\overline{\mathop{\mathrm{cd}}}(W_n)=\overline{\mathop{\mathrm{cd}}}(W)$. 
 
 The fact that $\lim_{n\rightarrow\infty}\mathop{\mathrm{cd}}(W_n)=\mathop{\mathrm{cd}}(W)$ follows verbatim.
\end{proof}

Next, we establish a connection between the crossing densities of graphons and the crossing densities of graphs.

\begin{theorem}\label{teo:G_W}
For every edge-weighted graph $G$ on $n$ vertices, we have that \[0\leq\overline{\mathop{\mathrm{cd}}}(W_G)-\overline{\mathop{\mathrm{cd}}}(G)\leq 1/n\text{ and } 0\leq\mathop{\mathrm{cd}}(W_G)-\mathop{\mathrm{cd}}(G)\leq 1/n.\]   
\end{theorem}

\begin{proof}
    Let $v_1,v_2,\dots,v_n$ denote the vertices of $G$ and $I_1,I_2,\dots,I_n$ be the corresponding subintervals of $[0,1]$ (see the definition of $W_G$ in Section~\ref{sec:cutdistance}). The argument is similar to the one in the proof of Theorem~\ref{teo:blowup}. We begin by showing the inequalities for rectilinear crossing densities.
  
    Given any $f\in\mathcal{F}$, we obtain a drawing $\mathcal{D}_f$ of $G$ as follows: For every $v_i$, choose uniformly at random an element $x_i$ from $I_i$; the point $f(x_i)$ will represent $v_i$. Draw every edge of $G$ as a segment joining its endpoints and, if necessary, perturb the vertices slightly so that no three of them lie on a line and no three edges have a common interior point. Note that the vertex perturbation step can be carried out without creating new crossings. A straightforward computation shows that \[\mathbb{E}[\overline{\mathop{\mathrm{cd}}}(G,\mathcal{D}_f)]\leq\overline{\mathop{\mathrm{cd}}}(W_G,f).\] Since this holds for any $f\in\mathcal{F}$, we get that $\overline{\mathop{\mathrm{cd}}}(G)\leq\overline{\mathop{\mathrm{cd}}}(W_G)$.
    
    Suppose now that we are given a rectilinear drawing $\mathcal{D}$ of $G$. This drawing induces an $f_\mathcal{D}\in\mathcal{F}$, which is defined by setting $f_\mathcal{D}(x)=p_i$, where $i$ is such that $x\in I_i$ and $p_i$ is the point that represents $v_i$ in $\mathcal{D}$. The value of $\overline{\mathop{\mathrm{cd}}}(f_\mathcal{D},W_G)$ is upper bounded by $\overline{\mathop{\mathrm{cd}}}(G)+1/n$, where the $1/n$ accounts for the crossings coming from the $4$-tuples of points in $[0,1]$ which contain two elements in the same $S_i$ (these $4$-tuples are similar to the bad $4$-tuples in the proof of Theorem~\ref{teo:blowup}). This yields $\overline{\mathop{\mathrm{cd}}}(W_G)\leq\overline{\mathop{\mathrm{cd}}}(G)+1/n$.

    The proof of the second part of the statement proceeds in essentially the same way. Let $\mathcal{D}\in\mathcal{D}_\mathcal{W}$ and construct a drawing $\mathcal{D}'$ of $G$ by placing $v_i$ at $\mathcal{D}(x)$, where $x$ is chosen uniformly at random from $I_i$, and then using the curve $\mathcal{D}(x_i,x_j)$ to represent the edge $(v_i,v_j)$. Again, we might have to tweak the drawing slightly so that no edges goes through a vertex, no two edges are tangent, and no three edges cross each other at the same point. As above, we have that \[\mathbb{E}[\mathop{\mathrm{cd}}(G,\mathcal{D}')]\leq\mathop{\mathrm{cd}}(W_G,\mathcal{D}),\] which implies $\mathop{\mathrm{cd}}(G)\leq\mathop{\mathrm{cd}}(W_G)$.

    Finally, for every drawing $\mathcal{D}$ of $G$ we obtain an element $\mathcal{D}'$ of $\mathcal{D}_\mathcal{W}$ in the following way: For every $x\in I_i$, let $\mathcal D'(x)$ be the point that represents $v_i$ in $\mathcal D$. For every $x,y\in[0,1]$ with $x\in I_i, y\in I_j$ and $i\neq j$, set $\mathcal{D}'(x,y)$ to be the curve that represents the edge $(v_i,v_j)$. As in the case of rectilinear crossing densities, this graphon drawing satisfies $\mathop{\mathrm{cd}}(W_G,\mathcal{D}')\leq\mathop{\mathrm{cd}}(G)+1/n$.
\end{proof}

Theorems~\ref{teo:estimable} and~\ref{teo:G_W} can now be combined to obtain the main result of this section.

\begin{theorem}\label{teo:continuous}
     The crossing density and the rectilinear crossing density are continuous with respect to the cut norm $d_\square$. Furthermore, if $G_1,G_2,\dots$ is a convergent sequence of edge-weighted graphs and the number of vertices of $G_n$ goes to infinity with $n$, then $\lim_{n\rightarrow\infty}\mathop{\mathrm{cd}}(G_i)$ exists and is equal to $\mathop{\mathrm{cd}}(W)$. The same is true with $\overline{\mathop{\mathrm{cd}}}$ in place of $\mathop{\mathrm{cd}}$.
    \end{theorem}

\begin{proof}
    Let $W\in\mathcal{W}_0$. We shall exhibit a sequence of graphs $G_1,G_2,\dots$ such that $W_{G_1},W_{G_2},\dots$ converges to $W$ and $\lim_{n\rightarrow\infty}\overline{\mathop{\mathrm{cd}}}(W_n)=\overline{\mathop{\mathrm{cd}}}(W),\ \lim_{n\rightarrow\infty}\mathop{\mathrm{cd}}(W_n)=\mathop{\mathrm{cd}}(W)$. Suppose for a moment that such a sequence exists. Then, it must be the case that $\overline{\mathop{\mathrm{cd}}}$ and the functional $\overline{g}$ mentioned in Theorem~\ref{teo:estimable} are one and the same, and this also true for $\mathop{\mathrm{cd}}$ and $g$. Since $\overline{g}$ and $g$ are continuous with respect to $d_\square$, so too are $\overline{\mathop{\mathrm{cd}}}$ and $\mathop{\mathrm{cd}}$. 
    
    In order to construct this graph sequence, we introduce some notation. For any $W\in\mathcal{W}$ and any partition $\mathcal{P}=\{S_1,S_2,\dots,S_n\}$ of the unit interval into measurable sets, let $W_\mathcal{P}^{\text{diag}}$ be the graphon such that $W_\mathcal{P}^{\text{diag}}(x,y)=W_\mathcal{P}(x,y)$ if $x$ and $y$ lie in different parts of $\mathcal{P}$, and $W_\mathcal{P}^{\text{diag}}(x,y)=0$ otherwise. The advantage that we gain from using $W_\mathcal{P}^{\text{diag}}$ instead of $W_\mathcal{P}$ is that if $\mathcal{P}$ is a partition of $[0,1]$ into intervals of the same length, then $W_\mathcal{P}^{\text{diag}}=W_G$ for some graph $G$. Note that if we let $\mathcal{P}_n$ be the partition into intervals of length $1/n$ for every positive integer $n$, then each pair of distinct points in $[0,1]$ is separated by all but a finite amount of the $P_i$'s. As a consequence, the sequence $W_{\mathcal{P}_1},W_{\mathcal{P}_2},\dots$ converges to $W$ almost everywhere. Observe that for every point $(x,y)\in[0,1]^2$ with $x\neq y$ there exists an integer $N$ such that $W_{\mathcal{P}_n}(x,y)=W_{\mathcal{P}_n}^{\text{diag}}(x,y)$ whenever $n\geq N$. This implies that the sequence $W_{\mathcal{P}_1}^\text{diag},W_{\mathcal{P}_2}^\text{diag},\dots$ converges to $W$ almost everywhere too. Let $G_1,G_2,\dots$ be edge-weighted graphs such that $W_{G_n}=W_{\mathcal{P}_n}^\text{diag}$ for every $n$. Then, $W_{G_1},W_{G_2},\dots$ converges to $W$ almost everywhere and Lemma~\ref{teo:pointwise} implies that $\lim_{n\rightarrow\infty}\overline{\mathop{\mathrm{cd}}}(W_n)=\overline{\mathop{\mathrm{cd}}}(W)$ and $\lim_{n\rightarrow\infty}\mathop{\mathrm{cd}}(W_n)=\mathop{\mathrm{cd}}(W)$. Since convergence almost everywhere implies convergence with respect to $d_\square$, the result follows.
\end{proof}

A graphon parameter $g$ is said to be \textit{invariant} if $g(W_1)=g(W_2)$ whenever $W_1$ and $W_2$ are weakly isomorphic. As an immediate consequence of the above theorem, we get that $\overline{\mathop{\mathrm{cd}}}$ and $\mathop{\mathrm{cd}}$ are invariant. 

\begin{corollary}\label{teo:invariant}
    The crossing density and the rectilinear crossing density of graphons are invariant. 
\end{corollary}

We point out that this corollary can also be deduced in a more direct manner without having to recur to crossing densities of graphs and Theorem~\ref{teo:estimable}.

\subsection{Some other properties of the crossing densities}

Let $W^p$ denote the constant graphon  defined by $W^p(x,y)=p$ for all $x,y\in[0,1]$. Theorem~\ref{teo:continuous} also has the following corollary, which ties together the crossing densities with some of the problems we briefly mentioned during the introduction.

\begin{corollary}\label{teo:W1}
    We have that \[\lim_{n\rightarrow\infty}\frac{\overline{\mathop{\mathrm{cr}}}(K_n)}{\binom{n}{4}}=24\lim_{n\rightarrow\infty}\overline{\mathop{\mathrm{cd}}}(K_n)=24\overline{\mathop{\mathrm{cd}}}(W^1)\] and, similarly \[\lim_{n\rightarrow\infty}\frac{\mathop{\mathrm{cr}}(K_n)}{\binom{n}{4}}=24\mathop{\mathrm{cd}}(W^1).\]
\end{corollary}

Clearly, if $c\geq0$ and $W_1$ and $W_2$ are graphons with $W_1=cW_2$, then $\mathop{\mathrm{cd}}(W_1)=c^2\cdot\mathop{\mathrm{cd}}(W_2)$ and $\overline{\mathop{\mathrm{cd}}}(W_1)=c^2\cdot\overline{\mathop{\mathrm{cd}}}(W_2)$. Crossing numbers of random graphs have received considerable attention (see, for example,~\cite{cr_random}). A sequence $G_1,G_2,\dots$ of edge-weighted graphs is said to be \textit{quasi-random with density} $p$ if the orders of the graphs go to infinity with $n$ and $W_{G_1},W_{G_2},\dots$ converges to $W^p$. There are many other equivalent (and often more practical) definitions of quasi-randomness (see~\cite{quasi-random,lovaszlimits}), but this one is the best suited for our purposes. The results in~\cite{approximatingrect} directly imply that the rectilinear crossing densities of any quasi-random sequence of graphs converge. We can show that the same is true for the crossing density. In fact, if we write $\overline{K}=\overline{\mathop{\mathrm{cd}}}(W^1)$ and $K=\mathop{\mathrm{cd}}(W^1)$, then $\mathop{\mathrm{cd}}(W^p)=Kp^2$ and $\overline{\mathop{\mathrm{cd}}}(W^p)=\overline{K}p^2$, and so Theorem~\ref{teo:continuous} yields the following.

\begin{corollary}
    Let $p\in[0,1]$. Then, for any sequence $G_1,G_2,\dots$ of edge-weighted graphs that is quasi-random with density $p$, we have that $\lim_{n\rightarrow\infty}\overline{\mathop{\mathrm{cd}}}(G_n)=\overline{K}p^2$ and $\lim_{n\rightarrow\infty}\mathop{\mathrm{cd}}(G_n)=Kp^2$.
\end{corollary}

Below, we obtain some additional properties of the crossing densities. The next observation is a direct consequence of the definitions of $\mathop{\mathrm{cd}}$ and $\overline{\mathop{\mathrm{cd}}}$.

\begin{observation}\label{teo:superaditive}
    Let $W$, $W_1$ and $W_2$ be graphons such that $W=W_1+W_2$. Then $\overline{\mathop{\mathrm{cd}}}(W)\geq\overline{\mathop{\mathrm{cd}}}(W_1)+\overline{\mathop{\mathrm{cd}}}(W_2)$ and $\mathop{\mathrm{cd}}(W)\geq\mathop{\mathrm{cd}}(W_1)+\mathop{\mathrm{cd}}(W_2)$.
\end{observation}

\begin{proof}
    We only prove the first inequality, as the second one can be deduced almost identically. For any $f\in\mathcal{F}$, we have that \[\overline{\mathop{\mathrm{cd}}}(W,f)=\frac{1}{8}\int_{[0,1]^4}[W_1(x_1,x_2)+W_2(x_1,x_2)][W_1(x_3,x_4)+W_2(x_3,x_4)]I_f(x_1,x_2,x_3,x_4)\prod_{i=1}^4 dx_i\]  
    \[\geq\frac{1}{8}\int_{[0,1]^4}[W_1(x_1,x_2)W_1(x_3,x_4)+W_2(x_1,x_2)W_2(x_3,x_4)]I_f(x_1,x_2,x_3,x_4)\prod_{i=1}^4 dx_i\] 
    \[=\overline{\mathop{\mathrm{cd}}}(W_1,f)+\overline{\mathop{\mathrm{cd}}}(W_2,f)\geq\overline{\mathop{\mathrm{cd}}}(W_1)+\overline{\mathop{\mathrm{cd}}}(W_2),\] and the result follows.
\end{proof}

This can be used to prove the following more interesting result.

\begin{theorem}
Let $W$ be a graphon and $\mathcal{P}=\{S_1,S_2,\dots,S_n\}$ a partition of the unit interval into measurable sets. Then, we have that $\overline{\mathop{\mathrm{cd}}}(W)\leq\overline{\mathop{\mathrm{cd}}}(W_\mathcal{P})$ and $\mathop{\mathrm{cd}}(W)\leq\mathop{\mathrm{cd}}(W_\mathcal{P}).$
\end{theorem}

\begin{proof}
    The technique below was also used in the proof of Proposition 14.13 in~\cite{lovaszlimits}.
    
    By Lemma~\ref{teo:invariant}, we may assume that each $S_i$ is an interval. For every $i$, let $a_i$ and $b_i$ be the endpoints of $S_i$ and let $\phi_i:S_i\rightarrow S_i$ be the measure preserving map defined by $\phi_i(x)=a_i+[2(x-a_i)(\mod b_i-a_i)]$. Let $\phi:[0,1]\rightarrow[0,1]$ be the measure preserving map defined by $\phi|_{S_i}=\phi_i$. The point of this is that the map from $S_i\times S_j$ to itself given by $\phi(x,y)=(\phi(x),\phi(y))$ is ergodic for every $i,j$. For every positive integer $n$, define a graphon $W_n$ with \[W_n(x,y)=\frac{1}{n}\sum_{t=0}^{n}W(\phi^t(x),\phi^t(y)).\] By Observation~\ref{teo:superaditive}, \[\overline{\mathop{\mathrm{cd}}}(W_n)\geq\frac{1}{n}\sum_{t=0}^{n}\overline{\mathop{\mathrm{cd}}}(W^{\phi^t})=\overline{\mathop{\mathrm{cd}}}(W).\] The sequence $W_1,W_2,\dots$ converges to $W_\mathcal{P}$ almost everywhere by the choice of $\phi$, hence $\overline{\mathop{\mathrm{cd}}}(W_\mathcal{P})=\lim_{n\rightarrow\infty}\overline{\mathop{\mathrm{cd}}}(W_n)\geq\overline{\mathop{\mathrm{cd}}}(W)$. The result for crossing densities follows from an analogous argument. 
\end{proof}

\begin{corollary}\label{teo:Wpmax}
    Let $p\in[0,1]$. Among all graphons $W$ with $||W||_1=\int_{[0,1]^2}W(x,y)\ dx\ dy=p$, $W^p$ achieves the largest crossing density, as well as the largest rectilinear crossing density.
\end{corollary} 

It might seem inconvenient that the definitions of $\mathcal{F}$ and $\mathcal{D}_\mathcal{W}$ allow for "degenerate" configurations. However, it is possible to add several more restrictions to these families without altering the crossing densities. We provide an example of this.

We say that a function $f\in\mathcal{F}$ is \textit{nice} if it is measure preserving and bounded.

\begin{theorem}\label{teo:generic}
Let $W$ be a graphon. Then, \[\overline{\mathop{\mathrm{cd}}}(W)=\inf_{f}\overline{\mathop{\mathrm{cd}}}(W,f),\] where $f$ ranges over all nice elements of $\mathcal{F}$.
\end{theorem}

\begin{proof}
    Combining the ideas in the proofs of theorems~\ref{teo:G_W} and~\ref{teo:continuous}, we can obtain a sequence $f_1,f_2,\dots$ of elements of $\mathcal{F}$, each of which arises from a rectilinear drawing of a graph (as did $f_\mathcal{D}$ in our proof that $\overline{\mathop{\mathrm{cd}}}(W_G)\leq\overline{\mathop{\mathrm{cd}}}(G)+1/n$ for every $n$ vertex graph $G$), and such that $\lim_{n\rightarrow\infty}\overline{\mathop{\mathrm{cd}}}(W,f_n)=\overline{\mathop{\mathrm{cd}}}(W)$. Each of these elements of $F$ maps each of the intervals in some partition of $[0,1]$ to a distinct point on the plane. We can rescale $f_n$ so that the points corresponding to the images of any two of these intervals are arbitrarily far away from each other. For every maximal subinterval of $[0,1]$ that is mapped to a single point by $f_n$, we can now reconstruct $f_n$ so that its restriction to the subinterval is instead measure preserving and its image is a disk. This modification can be carried out so that the rectilinear crossing density increases by an arbitrarily small amount. Repeating this process for every subinterval with the above property and then for every $n$, we arrive at a new sequence $f_1',f_2',\dots$ of nice elements of $\mathcal{F}$ such that $\lim_{n\rightarrow\infty}\overline{\mathop{\mathrm{cd}}}(W,f_n')=\overline{\mathop{\mathrm{cd}}}(W)$.
\end{proof}

We mentioned in the introduction that the crossing number of complete graphs is connected to Sylvester's four point problem. To finish this section, we describe how this connection is, in some sense, a consequence of the interplay between rectilinear crossing densities of graphons and rectilinear crossing numbers of graphs.

Given a bounded region $R\subset\mathbb{R}^2$ of positive measure, let $\lambda_R$ be the probability measure that arises from restricting the Lebesgue measure to $R$ and then normalizing. For any probability measure $P$ on the Lebesgue $\sigma$-algebra of $\mathbb{R}^2$, let $\text{c}(P)$ denote the probability that four points sampled independently from $P$ can be labeled by $a,b,c,d$ so that $I(a,b,c,d)=1$ (that is, $\text{conv}\{a,b\}\cap\text{conv}\{c,d\}\neq\emptyset$). It was shown by Scheinerman and Wilf~\cite{sylvester} (although it is worded in a slightly different manner in their paper) that \[\lim_{n\rightarrow\infty}\frac{\overline{\mathop{\mathrm{cr}}}(K_n)}{\binom{n}{4}}=\inf_{R}c(\lambda_R),\] where $R$ ranges over all bounded regions of positive measure. 

Now, for $f\in\mathcal{F}$, let $\lambda_f$ denote the corresponding pushforward probability measure on $\mathbb{R}^d$ (i.e., $\lambda_f(A)=\lambda(f^{-1}(A))$ for every measurable $A\subseteq\mathbb{R}^2$). Every probability measure of the form $\lambda_R$ corresponds to a standard probability space, and thus each $\lambda_R$ can be written as $\lambda_f$ for some nice $f\in\mathcal{F}$. Furthermore, if $f$ is nice then $\lambda_f=\lambda_R$ for some $R$. This implies that \[\inf_{f}c(\lambda_f)=\inf_{R}c(\lambda_R),\] where $f$ ranges over all nice elements of $\mathcal{F}$. On top of this, it is not hard to see that $24\overline{\mathop{\mathrm{cd}}}(W^1,f)=\text{c}(\lambda_f)$ for all nice $f$ (since the set $f^{-1}(\ell)\subset[0,1]$ has measure $0$ for every line $\ell$ on the plane). This observation, along with Theorem~\ref{teo:generic}, implies that \[24\overline{\mathop{\mathrm{cd}}}(W^1)=\inf_{R}c(\lambda_R).\] By the first part of Corollary~\ref{teo:W1}, this is actually equivalent to the result of Scheinerman and Wilf.

\section{Open problems}\label{sec:final}

\begin{itemize}
    \item Unlike the algorithm from~\cite{approximatingrect}, our drawing algorithm in Theorem~\ref{teo:main} is not deterministic.

\begin{problem}
    Is there a deterministic polynomial-time algorithm that for every $n$-vertex graph $G$ produces a drawing of $G$ with $\mathop{\mathrm{cr}}(G)+o(n^4)$ crossings?
\end{problem}

\item While Theorem~\ref{teo:continuous} shows that the graphon parameters $\mathop{\mathrm{cd}}$ and $\overline{\mathop{\mathrm{cd}}}$ are indeed the "correct" discrete analogs of the rectilinear crossing density and the crossing density of graphs, there are several ways in which we could have defined them. Indeed, as we exemplified at the end of the previous section, minor tweaks to the definitions of $\mathcal{F}$ and $\mathcal{D}_\mathcal{W}$ will often have no influence whatsoever on the graphon parameters $\mathop{\mathrm{cd}}$ and $\overline{\mathop{\mathrm{cd}}}$ (see Theorem~\ref{teo:generic}). Since the definition of $\mathcal{D}_\mathcal{W}$ is rather artificial (the set is made up of all the simple graphon drawings such that $I_\mathcal{D}$ measurable), it makes sense to ask whether this family can be substituted by a more natural set of functions from $[0,1]^2$ to $\mathcal{C}$ without altering $\mathop{\mathrm{cd}}$.

\item As mentioned near the end of Section~\ref{sec:densities}, we do not know whether for every graphon $W$ there is an $f\in\mathcal{F}$ such that $\overline{\mathop{\mathrm{cr}}}(W)=\overline{\mathop{\mathrm{cr}}}(W,f)$ or a $\mathcal{D}\in\mathcal{D}_\mathcal{W}$ such that $\mathop{\mathrm{cd}}(W)=\mathop{\mathrm{cd}}(W,\mathcal{D})$. We believe that the answer to the first of these questions is negative, and we make the following conjecture.

\begin{conjecture}
    There exists no $f\in \mathcal F$ such that $\overline{\mathop{\mathrm{cd}}}(W^1,\mathcal{D})=\overline{\mathop{\mathrm{cd}}}(W^1,f)$. In particular, the infimum in \[\lim_{n\rightarrow\infty}\frac{\overline{\mathop{\mathrm{cr}}}(K_n)}{\binom{n}{4}}=\inf_{R}c(\lambda_R)\] cannot be substituted by a minimum.
\end{conjecture}

\item The definition of a simple graphon drawing can also be modified to obtain continuous analogs of other kinds of crossing numbers. For example, if we get rid of condition $4$, we obtain a sort of \textit{pair crossing density}\footnote{The \textit{pair crossing density} of a graph $G$ is the least number of pairs of crossing edges over all drawings of $G$.} for graphons, while adding a fifth condition which requires that each $\mathcal{D}(x,y)$ is $x$-monotone results in a \textit{monotone crossing density}\footnote{A curve is $x$\textit{-monotone} if every vertical line intersects it in at most one point. A graph drawing is said to be \textit{monotone} if every edge is represented by an $x$-monotone curve. The \textit{monotone crossing number} of a graph is the least number of crossings that can be attained by a monotone drawing.}. Even though several of our arguments carry over to these variants, we do not know whether a result along the lines of Theorem~\ref{teo:closegraphs} holds for other kinds of crossing numbers.

\item For any $p\in[0,1]$, the set of graphons with $||W||_1=p$ is clearly closed with respect to $d_\square$. Since $(\widetilde{\mathcal{W}_0},\delta_\square)$ is compact and the crossing densities are invariant, the following problem is well posed.

\begin{problem}
    For every $p\in[0,1]$, find two graphon $W_1$ and $W_2$ with $||W_1||_1=||W_2||_1=p$ such that \[\overline{\mathop{\mathrm{cd}}}(W_1)=\min_{W\in\mathcal{W}_0,||W||_1=p}\overline{\mathop{\mathrm{cd}}}(W)\] and \[\mathop{\mathrm{cd}}(W_2)=\min_{W\in\mathcal{W}_0,||W||_1=p}\mathop{\mathrm{cd}}(W).\]
\end{problem}

Corollary~\ref{teo:Wpmax} provides an answer to the variant of this problem where the minimums are substituted by maximums. We expect both of the above questions to be difficult to answer, but we believe that doing so might be an important step towards determining the so-called \textit{midrange crossing constant}\footnote{Let $\kappa(n,e)$ denote the minimum crossing number among all $n$-vertex graphs $G$ with $e$ edges. It was shown by Pach, Spencer and Tóth~\cite{pach1999new} that there exists a constant $\kappa$ (the \textit{midrange crossing constant}) such that, under the assumption that $n\ll e\ll n^2$, $\lim_{n\rightarrow \infty}\kappa(n^2/e^3)=\kappa$. An analogous result holds for rectilinear crossing numbers.}.  

Let $\mathbb S^2$ denote a $2$-dimensional sphere in $\mathbb R^3$ with total surface area $1$. We will now describe a graphon defined on $\mathbb S^2$ instead of $[0,1]$, but we remark that it can be transformed to an equivalent graphon in $\mathcal W_0$ by means of a measure preserving transformation, and so the theory we have developed still applies to it. Fix $p\in[0,1]$ and let $\tau_p$ denote the unique real number such that, for any $x\in\mathbb S^2$, the intersection of $\mathbb S^2$ with the ball of radius $\tau_p$ centered at $x$ has surface area $p$. Define $W_{\mathbb S^2,p}:\mathbb S^2\times\mathbb S^2\rightarrow \mathbb [0,1]$ by setting $W_{\mathbb S^2,p}(x,y)=1$ if $x$ and $y$ are at distance at most $\tau_p$ and $W_{\mathbb S^2,p}(x,y)=0$ otherwise. Inspired by the main result of~\cite{czabarka2020some}, we make the following conjecture.

\begin{conjecture}
    For every $p\in[0,1]$, we have that \[\operatorname{cd}(W_{\mathbb S^2,p})=\min_{W\in\mathcal{W}_0,||W||_1=p}\mathop{\mathrm{cd}}(W).\]
\end{conjecture}

\item Finally, motivated by the Kuratowski-Wagner Theorem~\cite{kuratowski1930probleme,wagner1937eigenschaft} (which states that a graph is planar if and only if it doesn't contain a subdivision of $K_5$ or $K_{3,3}$ as a subgraph), we ask the following question.

\begin{problem}
    Let $W$ be a graphon. What can be said about $\operatorname{cd}(W)$ and $\overline{\operatorname{cd}}(W)$ just from knowing the homomorphism densities $t(F,W)$ where $F$ ranges over all subdivisions of $K_5$ and $K_{3,3}$?
\end{problem}

\end{itemize}

\bibliographystyle{plain}
\bibliography{refs}

\end{document}